\documentclass[a4paper,12pt]{article}

\usepackage{amsthm,amsmath,amscd,mathtools}
\usepackage[vmargin=2cm,hmargin=2cm,headheight=14.5pt,top=2cm,headsep=.5cm]{geometry}
\usepackage{hyperref}
\usepackage[utf8]{inputenc}
\usepackage[english]{babel}
\usepackage[autobold]{mathfixs}
\usepackage{stackrel}
\usepackage{cases}
\usepackage[charter]{mathdesign}
\usepackage{bm}
\usepackage{caption}
\usepackage{xcolor}
\usepackage{graphicx}
\DeclareMathSizes{12}{12}{8}{6}

\usepackage{cite}
\usepackage{enumitem}
\setlist[itemize,2]{label=$\centerdot$}
\setlist[itemize,3]{label=$\triangle$}

\usepackage[shortcuts]{extdash}

\newtheoremstyle{ptheorem}{1em}{0em}{\itshape}{}{\bfseries}{.}{.5em}{\thmname{#1}\thmnumber{
		#2}\thmnote{ (\hspace{-.01pt}{#3})}}

\theoremstyle{ptheorem}

\newtheorem{thm}{Theorem}[section]
\newtheorem{pro}[thm]{Proposition}
\newtheorem{lem}[thm]{Lemma}
\newtheorem{cor}[thm]{Corollary}

\newtheoremstyle{hdef}{1em}{0em}{}{}{\bfseries}{.}{.5em}{\thmname{#1}\thmnumber{
		#2}\thmnote{ (\hspace{-.01pt}{#3})}}
\theoremstyle{hdef}

\newtheorem{dfn}[thm]{Definition}
\newtheorem{rem}[thm]{Remark}

\newtheorem{exa}[thm]{Example}

\numberwithin{equation}{section}
\numberwithin{figure}{section}

\newcommand{\cB}{{\mathcal B}}
\newcommand{\cC}{{\mathcal C}}
\newcommand{\cD}{{\mathcal D}}

\newcommand{\cI}{{\mathcal I}}

\newcommand{\cK}{{\mathcal K}}

\newcommand{\cM}{{\mathcal M}}

\newcommand{\bC}{{\mathbb C}}

\newcommand{\bF}{{\mathbb F}}

\newcommand{\bN}{{\mathbb N}}

\newcommand{\bR}{{\mathbb R}}

\renewcommand{\l}{\lambda}
\newcommand{\e}{\varepsilon}

\renewcommand{\phi}{\varphi}

\newcommand{\n}{{n\in\bN}}

\renewcommand{\(}{\left(}
\renewcommand{\)}{\right)}

\newcommand{\pd}{\partial}

\newcommand{\olb}[1]{\vbox{\offinterlineskip\ialign{\hfil##\hfil\cr
			$\rotatebox[origin=c]{90}{$]$}$\cr\noalign{\kern-.45ex}{$#1$}\cr}}}

\newcommand{\noop}[1]{}

\renewcommand{\ss}{\subset}
\usepackage{stmaryrd}

\allowdisplaybreaks

\begin{document}

	\title{Constructive solutions of the heat equation\\with Stieltjes derivatives}

\author{Clara Senín$^{\dagger}$\\
	\normalsize \emph{e\--mail:} clara.senin.sanchez@rai.usc.gal\\
	F. Adri\'an F. Tojo$^{\dagger*}$ \\
	\normalsize \emph{e\--mail:} fernandoadrian.fernandez@usc.es\\
	\small \emph{$^{\dagger}$CITMAga, 15782, Santiago de Compostela, Spain}\\
	\small \emph{$^{*}$Departamento de Estatística, Análise Matemática e Optimización},\\
	\small \emph{Universidade de Santiago de Compostela, 15782, Santiago de Compostela, Spain.}
}

\date{}

\maketitle

\begin{abstract}
	In this work, we investigate the one\--dimensional heat equation within the framework of Stieltjes calculus. We first consider the equation associated with two fixed derivators and develop a constructive approach to establish the existence of solutions. We then study the corresponding initial value problem and incorporate several types of boundary conditions. Finally, we introduce a notion of multivariable derivator, suitable for higher\--dimensional settings, and obtain explicit solutions of the heat equation for relevant classes of such derivators.
\end{abstract}

\medbreak


\medbreak


\section{Introduction}

Stieltjes differential calculus provides a powerful framework for modeling dynamical systems that exhibit latency periods, abrupt transitions, or impulsive effects. By allowing differentiation with respect to functions of bounded variation, it unifies continuous and discrete calculus and offers a natural setting for the study of dynamic equations on time scales. In recent years, the Stieltjes derivative has been successfully applied to a wide range of models, including insect population dynamics \cite{velutina,biology}, solute dissolution and water evaporation processes \cite{1st_order_systems}, vehicle motion \cite{Resolution_methods}, prey\--predator systems with seasonal effects \cite{predator_prey}, and thermal stress in solar panels \cite{MaiaTo2}.

From a theoretical perspective, Stieltjes differential equations are particularly well suited for describing evolution processes with discontinuities in time. Within this framework, the Stieltjes derivative encodes both absolutely continuous dynamics and jump phenomena through the atomic components of the underlying measure, thereby providing a unified operator that naturally incorporates impulsive behavior.

In the setting of partial differential equations, where discontinuities in time interact with spatial diffusion, impulsive parabolic problems have been extensively studied; see, for example, \cite{HernandezM.2008,Hakl2017,Liu2011,Liu1998,Wang2010,Erbe1991,Rogovchenko1996,Dou2006,Kirane1997,Gao2003,Bainov1996,Chan1996,Rogovchenko1997,Vlasenko2008,Zhang2002}. These models describe diffusion processes subject to instantaneous perturbations or abrupt changes in the medium, and arise naturally in applications such as heat conduction with thermal shocks, phase transitions, and controlled diffusion processes. In this context, the formulation of parabolic equations via Stieltjes derivatives provides a unified and flexible alternative to classical impulsive models, avoiding the explicit imposition of jump conditions while preserving the essential features of the dynamics.

Despite this growing interest, the study of classical parabolic equations ---particularly the heat equation--- within the Stieltjes framework remains relatively limited. Existing contributions primarily rely on abstract existence results and functional analytic techniques, such as semigroup theory and diagonalization methods, as well as on numerical approaches for treating initial and boundary value problems.

For instance, in \cite{Fernandez2020}, a diagonalization method is employed to establish existence and uniqueness results for initial value problems subject to Dirichlet and Neumann boundary conditions, and Stieltjes\--Bochner spaces are introduced to handle the problem. In \cite{velutina}, the theory is applied to model the diffusion of an invasive species. Furthermore, parabolic equations within the framework of dynamic equations on time scales ---generalized via Stieltjes derivatives--- have been studied in \cite{cuchta_heat_2023,Bohner2024,Georgiev2025,Georgiev2025a,Jackson2006}.

Most available results, however, are restricted either to one\--dimensional or specific classes of derivators (i.e., functions generating the Stieltjes derivative), leaving open the development of constructive approaches and extensions to more general multidimensional derivators. Thus, the aim of this work is to develop a constructive solution theory for the heat equation with Stieltjes derivatives. More precisely, we employ the method of separation of variables to derive explicit series representations of solutions and to establish precise conditions under which these series define valid solutions. In contrast to semigroup and diagonalization approaches, our methodology yields explicit representations, facilitates the treatment of various boundary conditions, and naturally extends to multidimensional settings.

In particular, building upon \cite{periodic}, we derive solutions under periodic boundary conditions.
We also introduce a more general notion of derivator, better suited to higher\--dimensional problems, and incorporate it into the analysis of the heat equation within the Stieltjes framework. This notion constitutes a nontrivial extension of the framework of product time scales (cf.~\cite[Chapter~6]{bohner2016multivariable}), as it allows for interdependence among the measures associated with the different variables, rather than assuming a simple product structure.

The paper is organized as follows. In Section~\ref{Prelim}, we present the main definitions and fundamental results from Stieltjes calculus used throughout the paper. In Section~\ref{solgeneral}, we introduce the heat equation with Stieltjes derivatives and consider two types of derivators. We construct solutions via separation of variables and analyze the conditions under which the resulting series define solutions of the equation. In Section~\ref{IVP}, we incorporate initial conditions and establish existence results for the corresponding initial value problem under suitable assumptions. Section~\ref{Boundary_cond} is devoted to periodic, Dirichlet, and Neumann boundary conditions under appropriate constraints on the derivators. Finally, in Section~\ref{Two_derivators}, we introduce a multivariable derivator and study the heat equation with respect to a two\--variable derivator, obtaining several results and explicit solutions in representative cases.

\section{Preliminaries}\label{Prelim}

In this section, we present some basic results about Stieltjes differential calculus which will be needed later in this work. We will mainly follow \cite{velutina,Fernandez2022}.

Let $\mathbb{F}$ denote the field $\mathbb{R}$ or $\mathbb{C}$. Throughout this section, we consider a left\--continuous, nondecreasing function $g:\mathbb{R}\rightarrow \mathbb{R}$, which we refer to as \textit{derivator}. We define the following set:
\[D_{g}=\left\{t\in \mathbb{R}\ : \ g(t^{+})-g(t)>0\right\},\]
where $g(t^{+})$ represents the right\--hand side limit $\lim_{s\rightarrow t^{+}}g(s)$.
The \textit{jump of $g$ at $t$} is denoted by $\Delta g(t):=g(t^{+})-g(t)$, for each $t\in \mathbb{R}$. Since $g$ is nondecreasing, $D_g$ is countable.

Define
\[C_{g}=\left\{t\in \mathbb{R}\ : \ \text{there exists } \varepsilon >0 \text{ such that $g$ is constant on }(t-\varepsilon,t+\varepsilon)\right\}.\]
We have that $C_{g}$ is an open subset of $\mathbb{R}$ with respect to the usual topology $\tau_{u}$. Therefore, we can write
\begin{equation}\label{Cg}
	C_{g}=\bigcup_{n\in\widetilde{\Lambda}}(a_{n},b_{n}),
\end{equation}
where
$\widetilde{\Lambda}\subset \mathbb{N}$ and the intervals $(a_{n},b_{n})$ are pairwise disjoint.

Given a derivator $g:\mathbb{R}\rightarrow\mathbb{R}$, we consider the \textit{Lebesgue\--Stieltjes measure $\mu_g$ associated with $g$}, given by
\[\mu_{g}(\left[c,d\right))=g(d)-g(c), \quad c,d\in\mathbb{R},\ c<d,\]
and extended using Carathéodory's extension theorem.
We define the \textit{outer measure associated with $g$} as:
\begin{equation}\label{outer_m}
	\mu_{g}^{*}(A):=\inf\left\{\sum_{n=1}^{\infty}(g(b_{n})-g(a_{n}))\ : \ A\subset\bigcup_{n=1}^{\infty}[a_{n},b_{n}), \ \{[a_{n},b_{n})\}_{n=1}^{\infty}\subset \widetilde{\mathcal{C}} \right\}, \quad A\subset\mathbb{R},
\end{equation}
where
\[\widetilde{\mathcal{C}}=\left\{[a,b)\ : \ a,b\in\mathbb{R},\ a<b\right\}.\]
We consider the measure space $(\mathbb{R},\mathcal{M}_{g},\mu_{g})$, where $\cM_g$ and $\mu_g$
are defined analogously to the classical Lebesgue $\sigma$\--algebra and measure.
A set is said to be \textit{$g$\--measurable} if it is measurable with respect to $\mu_{g}$. Likewise, we say that a function is \textit{$g$\--measurable} if it is measurable with respect to $\mu_{g}$. We denote by $\mathcal{L}^{1}_{g}(X,\mathbb{F})$ the space of $\mu_{g}$\--integrable functions defined on the $g$\--measurable subset $X\subset \mathbb{R}$ taking values in $\mathbb{F}$. If $f\in \mathcal{L}^{1}_{g}(X,\mathbb{F})$, we denote its integral by
\[\int_{X}f(s)\operatorname{d}\mu_{g}(s)\equiv \int_{X}f\operatorname{d}\mu_{g}.\]

A property is said to hold \textit{$g$\--almost everywhere} in $X\subset \mathbb{R}$ if it holds at every point of $X$ except on a subset $N\subset X$ such that $\mu_{g}(N)=0$. We abbreviate the expression ``$g$\--almost everywhere'' by ``$g$\--a. e.''.

Just as with the classical Lebesgue integral, we can consider the equivalence relation on $\mathcal{L}_{g}^{1}(X,\mathbb{F})$ that identifies two functions which are equal $g$\--almost everywhere. We denote the resulting quotient space by $L_{g}^{1}(X,\mathbb{F})$.

We now introduce the concept of continuity with respect to a derivator.
\begin{dfn}
	Let $X\subset \mathbb{R}$, $(Y,\|\cdot\|_Y)$ a normed space and $f:X\rightarrow Y$. Given $t\in X$, the function $f$ is said to be \textit{$g$\--continuous at $t$} if for every $\varepsilon\in\mathbb{R}^{+}$ there exists $\delta\in \mathbb{R}^{+}$ such that $\| f(s)-f(t)\|_{Y}<\varepsilon$ whenever $s\in X$ and $|g(s)-g(t)|<\delta$. The function $f$ is said to be \textit{$g$\--continuous} if it is $g$\--continuous at every $t\in X$.
\end{dfn}

\begin{rem}
	In the case $g=\operatorname{Id}$, where $\operatorname{Id}:\mathbb{R}\rightarrow\mathbb{R}$ is the identity function, we recover the usual definition of continuity. However, unlike classically continuous functions, $g$\--continuous functions are not necessarily locally bounded \cite[Example 3.3]{1st_order_systems}.
\end{rem}

We denote by $\mathcal{C}_{g}(X,Y)$ the space of $g$\--continuous functions defined on the set $X$ and taking values in $Y$. We use the notation $\mathcal{BC}_{g}(X,Y)$ to refer to the space of $g$\--continuous functions that are bounded on $X$. It is shown in \cite[Theorem~3.4]{1st_order_systems} that, if $(Y,\|\cdot\|_Y)$ is a Banach space, so is $\mathcal{BC}_{g}(X,Y)$, endowed with the supremum norm,
\begin{equation*}
	\| h \|_{0}=\sup \limits_{t\in X}\|h(t)\|_{Y},\quad h\in\mathcal{BC}_{g}(X,Y).
\end{equation*}

Our next step is to define the Stieltjes derivative of a function with respect to a derivator $g$. Let $I\subset \mathbb{R}$ be an interval.

\begin{dfn}[{\cite[Definition 3.7]{Fernandez2022}}]\label{def_gderiv}
	Let $f:I\rightarrow Y$, where $Y$ is a Banach space, and fix $t\in I$. The \textit{Stieltjes derivative} of $f$ at $t$ (or the \textit{$g$\--derivative} of $f$ at $t$), is defined as
	\[f'_{g}(t)=
	\begin{dcases}
		\lim\limits_{s\rightarrow t}\frac{f(s)-f(t)}{g(s)-g(t)}, & t\notin D_{g}\cup C_{g}, \\
		\lim\limits_{s\rightarrow t^{+}}\frac{f(s)-f(t)}{g(s)-g(t)}, & t\in D_{g},\\
		\lim\limits_{s\rightarrow b_{n}^{+}}\frac{f(s)-f(b_{n})}{g(s)-g(b_{n})}, & t\in (a_{n},b_{n}),
	\end{dcases} \]
	if the respective limit exists. In this case, we say that $f$ is \textit{$g$\--differentiable at $t$}, or \textit{differentiable with respect to $g$ at $t$}.
	 We say that $f$ is \textit{$g$\--differentiable} if it is $g$\--differentiable at every point in its domain.	\end{dfn}

\begin{rem}
	For $t\in D_{g}$, the function $f$ is $g$\--differentiable at $t$ if and only if $f(t^{+})$ exists. In that case,
	\[f'_{g}(t)=\frac{f(t^{+})-f(t)}{\Delta g(t)}.\]
\end{rem}

\begin{rem}[{\cite[Remark 2.4]{second_non_constant_coef}}]\label{derivada_tstar}
	Let us introduce the notation
	\begin{equation}\label{tstar}
		t^{*}=		\begin{dcases}
			t, & t\notin C_{g}, \\
			b_{n}, & t\in (a_{n},b_{n})\subset C_{g}.
	\end{dcases} \end{equation}
	Then, the expression for the $g$\--derivative can be written more compactly as
	\[f'_{g}(t)=
	\begin{dcases}
		\lim\limits_{s\rightarrow t}\frac{f(s)-f(t)}{g(s)-g(t)}, & t\notin D_{g}\cup C_{g}, \\
		\lim\limits_{s\rightarrow t^{+}}\frac{f(s)-f(t^{*})}{g(s)-g(t^{*})}, & t\in D_{g}\cup C_{g}.
	\end{dcases} \]
\end{rem}

The next result presents some properties of the Stieltjes derivative, including linearity and the product rule.
\begin{pro}[{\cite[Proposition 3.9]{Fernandez2022}}]\label{propiedades_gderivada}
	Let $I\subset\mathbb{R}$ be an interval, $t\in I$ and let $f_{1},f_{2}:I\rightarrow\mathbb{F}$ be two $g$\--differentiable functions at $t$. Then, using $t^{*}$ as given in~ \eqref{tstar}, the following statements hold:
	\begin{enumerate}
		\item Given $\lambda_{1},\lambda_{2}\in \mathbb{F}$, the function $\lambda_{1}f_{1}+\lambda_{2}f_{2}$ is $g$\--differentiable at $t$ and
		\[(\lambda_{1}f_{1}+\lambda_{2}f_{2})'_{g}(t)=\lambda_{1}(f_{1})'_{g}(t)+\lambda_{2}(f_{2})'_{g}(t).\]
		\item The product $f_{1}f_{2}$ is $g$\--differentiable at $t$ and
		\[(f_{1}f_{2})'_{g}(t)=(f_{1})'_{g}(t)f_{2}(t^{*})+(f_{2})'_{g}(t)f_{1}(t^{*})+(f_{1})'_{g}(t)(f_{2})'_{g}(t)\Delta g(t^{*}).\]
	\end{enumerate}
\end{pro}

We now examine function spaces with specific regularity in the context of Stieltjes differential calculus.

\begin{dfn}[{\cite[Definition 3.11]{Fernandez2022}}] Let $I\subset \mathbb{R}$ be an interval, $Y$ a Banach space and $g:\mathbb{R}\rightarrow\mathbb{R}$ a derivator. Define $\mathcal{C}_{g}^{0}(I,Y):=\mathcal{C}_{g}(I,Y)$. For each $k\in\mathbb{N}$, define recursively
	\[\mathcal{C}_{g}^{k}(I,Y)=\left\{f\in \mathcal{C}_{g}^{k-1}(I,Y)\ :\ (f_{g}^{(k-1)})'_{g}\in\mathcal{C}_{g}(I,Y)\right\},\]
	where $f_{g}^{(0)}=f$ and $f_{g}^{(k)}=(f_{g}^{(k-1)})'_{g}$. Finally, define \[\mathcal{C}_{g}^{\infty}(I,Y)=\bigcap_{k\in \mathbb{N}}\mathcal{C}_{g}^{k}(I,Y).\]

Analogously,	define $\mathcal{BC}_{g}^{0}(I,Y)=\mathcal{BC}_{g}(I,Y)$ and
	for each $k\in \mathbb{N}$, define recursively	\[\mathcal{BC}_{g}^{k}(I,Y)=\left\{f\in \mathcal{BC}_{g}^{k-1}(I,Y)\ :\ f_{g}^{(k)}\in\mathcal{BC}_{g}(I,Y)\right\}.\]
\end{dfn}

Given $k\in \mathbb{N}$, we can equip the space $\mathcal{BC}_{g}^{k}(I,Y)$ with the norm
\[\|f\|_{k}:=\max\limits_{0\leq i\leq k}\|f_{g}^{(i)}\|_0.\]

We now present the Fundamental Theorem of Calculus for Stieltjes derivatives. Before doing so, we introduce the notion of absolute continuity with respect to a derivator, along with a result concerning the $g$\--differentiability of the integral of a $g$\--integrable function.
\begin{dfn}[{\cite[Definition 5.1]{NewUnification}}]
	Let $a, b\in \mathbb{R}$, with $a<b$, let $(Y,\|\cdot\|_Y)$ be a normed space and let $f:[a,b]\rightarrow Y$. We say that $f$ is \textit{$g$\--absolutely continuous on $[a,b]$} (or \textit{absolutely continuous with respect to $g$ on $[a,b]$}) if for every $\varepsilon\in \mathbb{R}^{+}$, there exists $\delta\in\mathbb{R}^{+}$ such that for every finite family of pairwise disjoint subintervals $\{(a_{n},b_{n})\}_{n=1}^{m}\subset [a,b]$ satisfying
	\[\sum\limits_{n=1}^{m}(g(b_{n})-g(a_{n}))<\delta,\]
	we have
	\[\sum\limits_{n=1}^{m}\|f(b_{n})-f(a_{n})\|_Y<\varepsilon.\]
	We denote by $\mathcal{AC}_{g}([a,b],Y)$ the space of $g$\--absolutely continuous functions on $[a,b]$ taking values in $Y$. It follows that $\mathcal{AC}_{g}([a,b],Y)\subset\mathcal{BC}_{g}([a,b],Y).$

	\begin{thm}[{\cite[Theorem~3.26]{tesis}}]\label{TFC2}
		Let $a,b\in\mathbb{R}$, with $a<b$ and let $f\in\mathcal{L}_{g}^{1}([a,b),\mathbb{F})$. Then, the function $F:[a,b]\rightarrow\mathbb{F}$ defined by
		\[F(t):=\int_{[a,t)}f(s)\operatorname{d}\mu_{g}(s),\]
		is well\--defined, $g$\--absolutely continuous on $[a,b]$, and satisfies
		\[F'_{g}(t)=f(t),\quad g-a. e. \quad t\in[a,b].\]
	\end{thm}

\end{dfn}
\begin{thm}[{\cite[Theorem~5.4]{NewUnification}}]\label{TFC1}			Let $a, b\in \mathbb{R}$, with $a<b$, and let $f:[a,b]\rightarrow \mathbb{F}$. Then, the following statements are equivalent:
	\begin{enumerate}
		\item The function $f$ is $g$\--absolutely continuous on $[a,b]$.
		\item The function $f$ satisfies the following three conditions:
		\begin{enumerate}
			\item The $g$\--derivative $f'_{g}(t)$ exists $g$\--a. e. in $ \left[a,b\right)$,
			\item $f'_{g}\in \mathcal{L}_{g}^{1}(\left[a,b\right),\mathbb{F}),$
			\item For every $t\in [a,b],$
			\[f(t)=f(a)+\int_{\left[a,t\right)}f'_{g}(s)\operatorname{d}\mu_{g}(s).\]
		\end{enumerate}
	\end{enumerate}
\end{thm}

We will use the following notation for integrals:
\[\int_{x}^{y}f(s)\operatorname{d}\mu_{g}(s):=\left\{\begin{aligned}
	&\int_{[x,y)} f(s), \operatorname{d}\mu_g(s),\; y\geq x,\\
	& -\int_{[y,x)} f(s)\, \operatorname{d}\mu_g(s),\; y< x.
\end{aligned}
\right.\]

Next, we introduce the $g$\--exponential map, adopting a more general definition than that given in \cite[Definition 4.5]{Fernandez2022}. This function plays a key role in solving linear differential equations with Stieltjes derivatives.\par

\begin{dfn}[{\cite[Definition 5.1]{MaiaTo2}}]\label{def_gexp}
	Let $a,b\in \mathbb{R}$, with $a<b$, and let $g:\mathbb{R}\rightarrow\mathbb{R}$ be a derivator. Suppose $p:[a,b]\rightarrow\mathbb{F}$ is a function such that $p\in \mathcal{L}_{g}^{1}(\left[a,b\right),\mathbb{F})$. Define the \textit{$g$\--exponential map associated with the map $p$} as the function $\exp_{g}(p;a,\cdot):[a,b]\rightarrow\mathbb{F}$, defined for each $t\in[a,b]$ by
	\begin{equation*}
		\exp_{g}(p;a,t)=\prod\limits_{s\in [a,t)\cap D_{g}}(1+p(s)\Delta g(s))\exp\left(\int_{[a,t)\backslash D_{g}}p(s)\operatorname{d}\mu_{g}(s)\right).
	\end{equation*}
\end{dfn}
\begin{thm}[{\cite[Theorem 5.8]{MaiaTo2}}]\label{orden1}
Let $h \in L^1_g([a,b),\mathbb{C})$.   Then, the function defined for every $t\in [a,b]$ by
\begin{equation*}
	x(t):=x_0 \exp_{g}(h;a,t),
\end{equation*}
is the unique $g$\--absolutely continuous solution of  the initial value problem
\begin{equation*}\label{eq:exp:linear eq}
	x_g'(t) = h(t)x(t) \quad \text{for $\mu_g$-almost every $t \in [a,b]$}, \quad x(a) =x_0.
\end{equation*}
Equivalently,
\[
x(t)= x_0+ \int_{[a,t)} h(s)x_0\exp_g(h;a,t)\, \operatorname{d}\mu_g(s) \quad \text{for every $t \in [a,b]$}.
\]
\end{thm}
\begin{dfn}
	We say that a map $p:[0,T]\rightarrow \mathbb{F}$ is \textit{$g$\--regressive} if $1+p(t)\Delta g(t)\neq 0$ for all $t\in \left[0,T\right)\cap D_{g}$. We say $p$ is \textit{strongly $g$\--regressive} if $1+p(t)\Delta g(t)> 0$ for all $t\in \left[0,T\right)\cap D_{g}$.
\end{dfn}
\begin{rem}
	Observe that in Definition~\ref{def_gexp} the function $p$ is not required to be $g$\--regressive, unlike in \cite[Definition 4.5]{Fernandez2022}. If $p$ is $g$\--regressive, then $\exp_{g}(p;a,\cdot)$ never vanishes. Furthermore, if $p$ is strongly $g$\--regressive, then $\exp_{g}(p;a,\cdot)$ is positive everywhere.
\end{rem}

When expressing solutions of the heat equation, linear combinations of $g$\--exponentials naturally arise. Therefore, it will be useful to introduce functions that serve as analogues of the classical sine and cosine functions, as well as the hyperbolic sine and hyperbolic cosine functions.

\begin{dfn}[$g$\--sine and $g$\--cosine, {\cite[Definition 4.8]{Fernandez2022}}]\label{sencos}
	Let $T\in \mathbb{R}^{+}$, $b\in \mathcal{L}_{g}^{1}([0,T],\mathbb{F})$. We define $\sin_{g}(b;0,t)$ and $\cos_{g}(b;0,t)$, as the first and second components, respectively, of the unique solution in $\mathcal{AC}_{g}([0,T],\mathbb{F}^{2})$ of the following linear system:
	\begin{equation}\label{gsincos}
	\begin{dcases}
		\begin{pmatrix} \sin_g(b;0,t) \\ \cos_g(b;0,t) \end{pmatrix}'_g(t) = \begin{pmatrix} 0 & b(t) \\ -b(t) & 0 \end{pmatrix} \begin{pmatrix} \sin_g(b;0,t) \\ \cos_g(b;0,t) \end{pmatrix}, \; g-a. e.\hspace{0.1cm} \, t \in [0,T), \\
		\sin_g(b;0,0)=0, \; \cos_g(b;0,0)=1.
		\end{dcases}\end{equation}
\end{dfn}

\begin{pro}[{\cite[Proposition 4.10]{Fernandez2022}}]\label{prop_sencos} Given $b\in \mathcal{L}_{g}^{1}([0,T],\mathbb{R})$, we have that
	\[\sin_{g}(b;0,t)=\frac{\exp_{g}(bi;0,t)-\exp_{g}(-bi;0,t)}{2i},\quad \cos_{g}(b;0,t)=\frac{\exp_{g}(bi;0,t)+\exp_{g}(-bi;0,t)}{2}.\]
\end{pro}
\begin{dfn}[hyperbolic $g$\--sine and $g$\--cosine]\label{shch}
	Let $T\in \mathbb{R}^{+}$, $a\in \mathcal{L}_{g}^{1}([0,T],\mathbb{F})$. We define $\sinh_{g}(a;0,t)$ and $\cosh_{g}(a;0,t)$, as the first and second components, respectively, of the unique solution in $\mathcal{AC}_{g}([0,T],\mathbb{F}^{2})$ of the following linear system:
	\begin{equation} \label{gsinhcosh}
	\begin{dcases}
		\begin{pmatrix} \sinh_g(a;0,t) \\ \cosh_g(a;0,t) \end{pmatrix}'_g(t) = \begin{pmatrix} 0 & a(t) \\ a(t) & 0 \end{pmatrix} \begin{pmatrix} \sinh_g(a;0,t) \\ \cosh_g(a;0,t) \end{pmatrix}, \; g-a. e. \hspace{0.1cm} \, t \in [0,T), \\
		\sinh_g(a;0,0)=0, \; \cosh_g(a;0,0)=1.
		\end{dcases}\end{equation}
\end{dfn}

\begin{rem}
	The uniqueness of solution of systems \eqref{gsincos} and \eqref{gsinhcosh} follows from \cite[Theorem~7.3]{1st_order_systems}, taking $L=|b|$ and $L=|a|$, respectively.
\end{rem}

\begin{pro}\label{prop_shch} Given $a\in \mathcal{L}_{g}^{1}([0,T],\mathbb{F})$, the following identities hold:
	\[\sinh_{g}(a;0,t)=\frac{\exp_{g}(a;0,t)-\exp_{g}(-a;0,t)}{2},\quad\cosh_{g}(a;0,t)=\frac{\exp_{g}(a;0,t)+\exp_{g}(-a;0,t)}{2}.\]
\end{pro}
\begin{proof}
	Define the functions
	\[\phi(t)=\frac{\exp_{g}(a;0,t)-\exp_{g}(-a;0,t)}{2},\quad\psi(t)=\frac{\exp_{g}(a;0,t)+\exp_{g}(-a;0,t)}{2}.\]
	Since $\phi$ and $\psi$ are linear combinations of $g$\--absolutely continuous functions, we have that  $(\phi,\psi)\in\mathcal{AC}_{g}([0,T],\mathbb{F}^2)$. We will show that the pair $(\phi,\psi)$ solves the system \eqref{gsinhcosh}. Using the properties of the $g$\--exponential map and the linearity of the Stieltjes derivative (Proposition ~\ref{propiedades_gderivada}), we obtain
	\[\phi'_{g}(t)=\frac{a(t)\exp_{g}(a;0,t)+a(t)\exp_{g}(-a;0,t)}{2}=a(t)\psi(t),\]
	\[\psi'_{g}(t)=\frac{a(t)\exp_{g}(a;0,t)-a(t)\exp_{g}(-a;0,t)}{2}=a(t)\psi(t).\]
	Moreover, the initial conditions are satisfied:
	\[\phi(0)=\frac{\exp_{g}(a;0,0)-\exp_{g}(-a;0,0)}{2}=0,\quad\psi(0)=\frac{\exp_{g}(a;0,0)+\exp_{g}(-a;0,0)}{2}=1.\]
	By uniqueness of solution of \eqref{gsinhcosh}, the equalities in the statement follow.
\end{proof}

In the classical setting, the exponential map, along with the sine and cosine functions, coincides with its Taylor series expansion over the entire domain. In this context, it will be of interest to explore analogous results for the Stieltjes derivative.\par

Classical Taylor series are expressed as sums of monomials of the form $a_{n}x^{n}$, where $n\in\mathbb{N}$. To develop a similar expansion in the framework of Stieltjes calculus, we begin by introducing the $g$\--monomials, which serve as counterparts to the classical polynomials $x^{n}$.

\begin{dfn}[{\cite[Definition 3.1]{analytic}}]\label{g_monomials}
	Let $g:\mathbb{R}\rightarrow\mathbb{R}$ be a derivator, and fix a point $x_{0}\in \mathbb{R}$. We define $g_{x_{0},0}(x)=1$, for every $x\in \mathbb{R}$. For each $n\in\mathbb{N}$, we define $g_{x_{0},n}:\mathbb{R}\rightarrow\mathbb{R}$ recursively by
	\[  g_{x_{0},n}=
		n\int_{x_{0}}^xg_{x_{0},n-1}\operatorname{d}\mu_{g}.\]
	These functions are known as \textit{$g$\--monomials centered at $x_{0}$}, and the point $x_{0}$ is the \textit{center} of $g_{x_{0},n}$. The linear combinations of $g$\--monomials centered at $x_{0}$ are called \textit{$g$\--polynomials centered at $x_{0}$}.
\end{dfn}

To simplify notation, we will write $g_{n}\equiv g_{x_{0},n}$ for a $g$\--monomial centered at a specific $x_{0}\in \mathbb{R}$.\par

We now introduce the notion of a Stieltjes\--analytic function.
\begin{dfn}[{\cite[Definition 4.11]{analytic}}]
	Let $g:\mathbb{R}\rightarrow\mathbb{R}$ be a derivator. A function $f:\Omega\rightarrow\mathbb{F}$, where $\Omega\subset\mathbb{R}$ is open, is said to be \textit{Stieltjes\--analytic} in $\Omega$ if, for every $y\in \Omega$, there exist $\delta>0$, $t\in \mathbb{R}$ such that $y\in(t,t+\delta)$, and there also exists a sequence $\{a_{n}\}_{n=0}^{\infty}\subset\mathbb{F}$ satisfying
	\[f(x)=\sum\limits_{n=0}^{\infty}a_{n}g_{t,n}(x),\]
	for every $x\in (t,t+\delta)\subset\Omega$, and, moreover, the convergence is absolute.
\end{dfn}

In the classical case, the exponential map admits the power series expansion
\[\exp(z)=\sum\limits_{n=0}^{\infty}\frac{z^{n}}{n!}, \quad z\in\mathbb{C}.\]
This leads us to the question of whether the $g$\--exponential map can also be represented as a series of $g$\--monomials. To address this, we introduce some definitions and results from \cite{analytic}. In the next results, $\lambda$ represents an element from $\mathbb{F}\backslash \{0\}$.

\begin{dfn}[{\cite[Definition 5.2]{analytic}}]
	Given $x_{0}\in\mathbb{R}$, we define the \textit{exponential series associated with $g$ centered at $x_{0}$} as the function given by
	\begin{equation}\label{serie_exp}
		\exp_{g}(\lambda;x_{0})(x)=\sum\limits_{n=0}^{\infty}\lambda^{n}\frac{g_{x_{0},n}(x)}{n!}
	\end{equation}
	on the set of points where the series converges uniformly.
\end{dfn}

Following the notation in \cite{analytic}, we denote by $\Omega_{x_{0}}$ the maximal interval on which the series \eqref{serie_exp} is absolutely convergent.

\begin{pro}[{\cite[Proposition 5.3]{analytic}}]
	Let $g:\mathbb{R}\rightarrow\mathbb{R}$ be a derivator, and fix $x_{0}\in\mathbb{R}$. If $\Delta g(x)<|\lambda|^{-1}$ for every $x<x_{0}$, then $\Omega_{x_{0}}=\mathbb{R}$.
\end{pro}

Analogously to the classical case, we are interested in whether the expression \eqref{serie_exp} defines a Stieltjes\--analytic function on the real line.

\begin{thm}[{\cite[Theorem~5.9]{analytic}}]
	Let $g$ be a derivator and fix $x_{0}\in\mathbb{R}$. If $1+\lambda\Delta g(x)\neq 0$ for all $x<x_{0}$, then, there exists a Stieltjes\--analytic extension of $\exp_{g}(\lambda;x_{0})$ to the real line.

	Moreover, assuming $x_{0}=0$ and choosing $t<\min\{x,0\}$, this extension is given by
	\[\operatorname{Exp}_{g}(\lambda;0)(x):=\frac{\exp_{g}(\lambda;t)(x)}{\exp_{g}(\lambda;t)(0)}.\]
\end{thm}

We now explore the role of the function $\operatorname{Exp}_{g}$ in solving the first\--order linear equation.

\begin{cor}[{\cite[Corollary 5.12 ]{analytic}}]
	If $1+\lambda\Delta g(x)\neq 0$ for all $x<x_{0}$, then $\operatorname{Exp}_{g}(\lambda;x_{0})$ is a Stieltjes\--analytic solution defined on $\mathbb{R}$ of the problem
	\begin{equation*}\begin{dcases}
			v'_{g}(x)=\lambda v(x), & x\in\mathbb{R},\\
			v(x_{0})=1.
		\end{dcases}
	\end{equation*}
\end{cor}

In \cite[Remark 5.15]{analytic}, it is stated that $\exp_{g}(\lambda;0)$ and the exponential function from Definition~\ref{def_gexp}, namely $\exp_{g}(p;0,\cdot)$ with $p=\lambda$, coincide for all $x\geq 0$.
Therefore, for $x\geq 0$ and writing $g_{0,n}\equiv g_{n}$, we have
\begin{equation}\label{igualdad}
	\exp_{g}(\lambda;0,x)=\sum\limits_{n=0}^{\infty}\lambda^{n}\frac{g_{n}(x)}{n!}.
\end{equation}
In addition, Proposition~\ref{prop_sencos}, implies that, for $x\geq0$,
\begin{equation}\label{serie_sen}
	\sin_{g}(\lambda;0,x)=\frac{1}{2i}\sum\limits_{n=0}^{\infty}\left((\lambda i)^{n}\frac{g_{n}(x)}{n!}-(-\lambda i)^{n}\frac{g_{n}(x)}{n!}\right)=\sum\limits_{n=0}^{\infty}(-1)^{n}\lambda^{2n+1}\frac{g_{2n+1}(x)}{(2n+1)!},
\end{equation}
and
\begin{equation}\label{serie_cos}
	\cos_{g}(\lambda;0,x)=\frac{1}{2}\sum\limits_{n=0}^{\infty}\left((\lambda i)^{n}\frac{g_{n}(x)}{n!}+(-\lambda i)^{n}\frac{g_{n}(x)}{n!}\right)=\sum\limits_{n=0}^{\infty}(-1)^{n}\lambda ^{2n}\frac{g_{2n}(x)}{(2n)!}.
\end{equation}
Using \eqref{igualdad} and Proposition~\ref{prop_shch}, for $x\geq 0$, it follows that

\begin{equation*}
	\sinh_{g}(\lambda;0,x)=\frac{1}{2}\sum\limits_{n=0}^{\infty}\left(\lambda^{n}\frac{g_{n}(x)}{n!}-(-\lambda)^{n}\frac{g_{n}(x)}{n!}\right)=\sum\limits_{n=0}^{\infty}\lambda^{2n+1}\frac{g_{2n+1}(x)}{(2n+1)!},
\end{equation*}
and
\begin{equation*}
	\cosh_{g}(\lambda;0,x)=\frac{1}{2}\sum\limits_{n=0}^{\infty}\left(\lambda^{n}\frac{g_{n}(x)}{n!}+(-\lambda )^{n}\frac{g_{n}(x)}{n!}\right)=\sum\limits_{n=0}^{\infty}\lambda ^{2n}\frac{g_{2n}(x)}{(2n)!}.
\end{equation*}

We now introduce a result regarding the $g$\--continuity and the $g$\--integrability of the uniform limit to a sequence of functions. The following lemma allows us to relate the integral of the limit function to the integrals of the functions in the sequence. It is similar to \cite[Lemma 3]{series_time_scales}.
\begin{lem}\label{intercambio_integrales}
Consider the operator $\cI:\cB\cC_g([a,b],\bF)\to\cB\cC_g([a,b],\bF)$ such that \[(\cI f)(t)=\int_{a}^{t}f\operatorname{d}\mu_{g}.\] $\cI$ is well defined and continuous. In particular, if $\{f_{n}\}_{n\in\mathbb{N}}\ss\cB\cC_g([a,b],\bF)$ is a sequence converging uniformly to $f\in\cB\cC_g([a,b],\bF)$, then
	\begin{equation*}\label{lim_integral}
		\int_{a}^{b}f\operatorname{d}\mu_{g}=\lim\limits_{n\rightarrow\infty} \int_{a}^{b}f_{n}\operatorname{d}\mu_{g}.
	\end{equation*}
\end{lem}

\begin{proof}
Given $f\in\cB\cC_g([a,b],\bF)$, we know, from Theorem~\ref{TFC1}, that $\cI f$ is bounded and $g$\--continuous, so $\cI$ is well defined.

Take a sequence $\{f_{n}\}_{n\in\mathbb{N}}\ss\cB\cC_g([a,b],\bF)$ converging to $f\in\cB\cC_g([a,b],\bF)$ (uniformly). If 	 $g(b)=g(a)$, $\mu_g([a,b))=0$ and any integral on $[a,b)$ is zero, so $\cI f_n\to \cI f$ uniformly.

On the other hand, if 	$g(b)>g(a)$
	and $\varepsilon>0$, since $\{f_{n}\}_{n\in\mathbb{N}}$ converges uniformly to $f$ on $[a,b]$, there exists $N\in \mathbb{N}$ such that for all $n\geq N$, $n\in\mathbb{N}$, and all $t\in[a,b]$,
	\[|f_{n}(t)-f(t)|<\frac{\varepsilon}{g(b)-g(a)}.\]
	Hence, using the linearity of the integral (see \cite[Proposition 2.3.1]{measure}) and the fact that the absolute value of the integral is bounded by the integral of the modulus (see \cite[Proposition 2.6.7]{Cohn}), we obtain
	\[\left|(\cI f_n)(t) -(\cI f)(t)\right|\leq \int_{a}^{b}\left\vert f_{n}-f\right\vert\operatorname{d}\mu_{g}<\frac{\varepsilon}{g(b)-g(a)}(g(b)-g(a))=\varepsilon,\]
so $\|\cI f_n- \cI f\|_0<\e$ and, thus, $\cI f_n\to \cI f$. This shows that $\cI$ is continuous.
\end{proof}

\section{General solution of the heat equation with Stieltjes derivatives and two derivators}\label{solgeneral}
Here we use the classical Fourier method to solve the heat equation within the Stieltjes setting.\par
Consider two derivators $g,h:\bR\to\bR$. The function $g$ will be associated with the variable $t$, and $h$ with the variable $x$. We use the notation $\pd_gu(t,x)=u'_g(\cdot,x)|_t$, $\pd_hu(t,x)=u'_h(t,\cdot)|_x$. Now consider the equation
\begin{equation} \label{calor}
	\pd_gu(t,x)-c^2\pd_h^2u(t,x)=0,
\end{equation}
where $c>0$. We assume that the variables $(t,x)$ take values in $[0,T]\times[0,L]$, for some given $T,L\in\mathbb{R}^{+}$. Throughout the rest of the paper, we assume $T\notin D_g\cup C_g$ and $L\notin D_h\cup C_h$, so that the corresponding derivatives can be computed at these points.

We begin by seeking a general solution of equation \eqref{calor}.
To this end, we apply the method of separation of variables (see \cite{strauss_partial_2007}, for example). That is, we look for a solution of \eqref{calor} of the form $u(t,x)=w(t)v(x)$.
 \begin{dfn}
 	We say that a function of two variables $u(t,x)$ is a \textit{function of separated variables} if it can be written as $u(t,x)=w(t)v(x)$ for some functions $w$ and $v$.
 \end{dfn}
  A function $u(t,x)=w(t)v(x)$ satisfies \eqref{calor} if and only if the following equation holds:
\begin{equation}\label{asterisco}
	w_{g}'(t)v(x)-c^{2}w(t)v_{h}''(x)=0.
\end{equation}
Since the zero function $u=0$ is always a solution of \eqref{calor}, we focus on finding nontrivial solutions.

\begin{thm}\label{form_sol}
	A function of the form $u(t,x)=w(t)v(x)$ is a nontrivial solution of \eqref{calor} if and only if there exists $\lambda\in\mathbb{F}$ such that $w$ is a nontrivial solution of
	\begin{equation}\label{p1}
		w_{g}'(t)=\lambda c^{2} w(t), \quad t\in [0,T],
	\end{equation}
	and $v$ is a nontrivial solution of
	\begin{equation}\label{p2}
		v_{h}''(x)=\lambda v(x), \quad x\in [0,L].
	\end{equation}
\end{thm}
\begin{proof}
	Assume first $u(t,x)=w(t)v(x)$ is a nontrivial solution of \eqref{calor}, that is, it satisfies equation \eqref{asterisco}. Then, there exists $(t_{0},x_{0})\in[0,T]\times[0,L]$ such that $u(t_{0},x_{0})\neq0$. Hence, $w(t_{0})\neq0$ and $v(x_{0})\neq0$. Evaluating \eqref{asterisco} at $(t,x)=(t_{0},x_{0})$ gives the equality
	\[\frac{v_{h}''(x_{0})}{v(x_{0})}=\frac{w_{g}'(t_{0})}{c^{2}w(t_{0})}=:\l.\]
	Now, evaluate \eqref{asterisco} at $t=t_{0}$ to obtain
	\[w_{g}'(t_{0})v(x)-c^{2}w(t_{0})v_{h}''(x)=0, \quad x\in [0,L], \]
which is equivalent to equation~\eqref{p2}.
	Similarly, evaluating \eqref{asterisco} at $x=x_{0}$ yields that
	\[w_{g}'(t)v(x_{0})-c^{2}w(t)v_{h}''(x_{0})=0, \quad t\in [0,T], \]
which is equivalent to equation~\eqref{p1}.
	Therefore, if $u(t,x)=w(t)v(x)$ is a nontrivial solution of \eqref{calor}, then there exists $\lambda\in \mathbb{F}$ such that $w$ solves \eqref{p1} and $v$ is a solution of \eqref{p2}, and both $v,w$ are nontrivial functions.

	Conversely, assume that there exist $\lambda\in\mathbb{F}$ and nontrivial functions $w$ and $v$ solving \eqref{p1} and \eqref{p2}, respectively.  Define the nontrivial function $u(t,x)=w(t)v(x)$. Then:
	\[		\pd_gu(t,x)-c^2\pd_h^2u(t,x)=w'_{g}(t)v(x)-c^{2}w(t)v''_{h}(x)=\lambda c^{2}w(t)v(x)-c^{2}w(t)\lambda v(x)=0.\]
	Thus, $u$ solves \eqref{calor}.
\end{proof}

Now we can focus on solving equations \eqref{p1} and \eqref{p2}. We begin by solving \eqref{p1}. To do so, we apply Theorem~\ref{orden1}. Consider the function
$p\in \mathcal{L}_{g}^{1}(\left[0,T\right),\mathbb{F})$ defined by $p(t)=\lambda c^{2}$. The general solution of \eqref{p1} is given by
\begin{equation*}\label{w}
	\begin{aligned}
	w(t)= & x_{0}\exp_{g}(p;0,t)=
	x_{0}\prod_{s\in [0,t)\cap D_{g}}(1+p(s)\Delta g(s))\exp\left(\int\limits_{[0,t)\backslash D_{g}}p(s)\operatorname{d}\mu_{g}(s)\right)\\= &
	x_{0}\prod_{s\in [0,t)\cap D_{g}}(1+\lambda c^{2}\Delta g(s))\exp\(\lambda c^{2}\mu_g\([0,t)\backslash D_{g}\)\).
	\end{aligned}
\end{equation*}

\begin{rem}
	We note that if $\lambda \geq0$, then the condition $1+\lambda c^{2}\Delta g(t)\neq 0$ automatically holds and $p$ is $g$\--regressive. Moreover, if there exists $\widetilde{t}\in D_{g}$ such that $1+\lambda c^{2}\Delta g(\widetilde{t})=0$, then the problem
	\begin{equation*}
		\begin{dcases}
			w_{g}'(t)=\lambda c^{2} w(t), & t\in (0,T),\\
			w(0)=x_{0},
		\end{dcases}
	\end{equation*}
	still admits a unique solution, although this solution vanishes from some point onward (see \cite[Remark 4.4]{Fernandez2022}).
\end{rem}

Next, we analyze problem \eqref{p2}. By using \cite[Theorem~5.3]{Fernandez2022}, we deduce that if $\lambda_{1}, \lambda_{2}\in \mathbb{C}$ are the roots of the characteristic equation $\lambda^{2}+P\lambda+Q=0$, the general solution of \[	v_{g}''(t)+Pv_{g}'(t)+Q=0,\]
takes the form:
\begin{enumerate}
	\item If $\lambda_{1}\neq \lambda_{2}$, then, 	for some constants $c_{1},c_{2}\in \mathbb{R}$,
	\begin{equation*}
		v(t)=c_{1}\exp_{g}(\lambda_{1};0,t)+c_{2}\exp_{g}(\lambda_{2};0,t).
	\end{equation*}
	\item If $\lambda_{1}=\lambda_{2}=\lambda$, then, for some constants $c_{1},c_{2}\in \mathbb{R}$,
	\begin{equation*}
		v(t)=c_{1}\exp_{g}(\lambda;0,t)+c_{2}\exp_{g}(\lambda;0,t)\int_{\left[0,t\right)}\frac{1}{1+\lambda \Delta^{+}g(s)}\operatorname{d}\mu_{g}(s).
	\end{equation*}
\end{enumerate}

In the particular case of \eqref{p2}, the characteristic equation is $X^{2}-\lambda=0$, with roots $\sqrt{\lambda}$ and $-\sqrt{\lambda}$. These are equal only if $\lambda=0$. Therefore, the general solution of \eqref{p2} is:
\begin{enumerate}
	\item If $\lambda=0$, then
	\[v(x)=c_{1}+c_{2}\int_{[0,x)}1\operatorname{d}\mu_{h}(s)=c_{1}+c_{2}(h(x)-h(0)).\]
	\item If $\lambda\neq 0$, then
	\[v(x)=c_{1}\exp_{h}(\sqrt{\lambda};0,x)+c_{2}\exp_{h}(-\sqrt{\lambda};0,x).\]
\end{enumerate}

Furthermore, if $\lambda>0$, using Definition~\ref{shch}, we may write
\[v(x)=c_{1}\sinh_{h}(\sqrt{\lambda};0,x)+c_{2}\cosh_{h}(\sqrt{\lambda};0,x),\]
while if $\lambda<0$, using Definition~\ref{sencos}, we may write
\[v(x)=c_{1}\sin_{h}(\sqrt{-\lambda};0,x)+c_{2}\cos_{h}(\sqrt{-\lambda};0,x).\]

We now combine the previous results to obtain a solution via separation of variables to the heat equation.

\begin{thm}\label{gen_sol}
	Let $m\in\mathbb{N}$, and let $\{\lambda_{n}\}_{n=1}^{m}\subset \mathbb{F}$ be a finite collection of numbers. Define, for each $n\in\{1,2,\ldots,m\}$ and $t\in[0,T]$, $p_{n}(t):=\lambda_{n}c^{2}$,  $t\in[0,T]$. Then, for any $\{a_{n}\}_{n=0}^{m}, \{b_{n}\}_{n=0}^{m}\subset \mathbb{F}$, the function $u:[0,T]\times[0,L]\rightarrow\mathbb{C}$, defined by
\[
		u(t,x)=a_{0}+b_{0}h(x)+\sum\limits_{n=1}^{m}\left(a_{n}\exp_{g}(p_{n};0,t)\exp_{h}(\sqrt{\lambda_{n}};0,x)+b_{n}\exp_{g}(p_{n};0,t)\exp_{h}(-\sqrt{\lambda_{n}};0,x)\right),
\]
	is a solution of \eqref{calor}.
\end{thm}

\begin{proof}
	Let $\lambda\in\mathbb{F}$ and define $p(t):=\lambda c^{2}$, $t\in[0,T]$. By multiplying the solutions obtained for each of the problems \eqref{p1} and \eqref{p2}, and applying Theorem~\ref{form_sol}, we obtain that the functions
	\[u(t,x)=x_{0}\exp_{g}(0;0,t)(c_{1}+c_{2}(h(x)-h(0)))=x_{0}(c_{1}+c_{2}(h(x)-h(0))),\]
	corresponding to the case $\lambda=0$, and
	\[u(t,x)=x_{0}\exp_{g}(p;0,t)(c_{1}\exp_{h}(\sqrt{\lambda};0,x)+c_{2}\exp_{h}(-\sqrt{\lambda};0,x)),\]
	otherwise are solutions of equation \eqref{calor}.
	Collecting constants, we can write these solutions as  $u(t,x)=a_{0}+b_{0}h(x)$
	and
	\[u(t,x)=a\exp_{g}(p;0,t)\exp_{h}(\sqrt{\lambda};0,x)+b\exp_{g}(p;0,t)\exp_{h}(-\sqrt{\lambda};0,x),\]
	respectively, for constants $a_{0},b_{0},a,b\in\mathbb{R}$.
	Since \eqref{calor} is linear and homogeneous, the principle of superposition yields that any linear combination of such solutions is a again a solution of \eqref{calor}.
\end{proof}

\subsection{Series solution}\label{sol_serie}

It is natural to ask what conditions must be imposed for an infinite sum of solutions of separated variables of \eqref{calor} to remain a solution of the equation.
First, we present a result that guarantees the $g$\--differentiability of the limit of a sequence of $g$\--differentiable functions. To this end, we adapt \cite[Lemma 4]{series_time_scales}.

\begin{pro}\label{prop3.2}
	The operator $\cD: \cB\cC_g^1([a,b],\bF)\to \cB\cC_g([a,b],\bF)$ such that $\cD f:=f'_g$ is continuous. In particular, if $\{f_{n}\}_{n\in\mathbb{N}}\ss\cB\cC_g^1([a,b],\bF)$ converges to $f$ in $\cB\cC_g^1([a,b],\bF)$, then $(f_n)'_g\to f_g'$ in $\cB\cC_g([a,b],\bF)$.
\end{pro}

\begin{proof}
    We consider in $\cB\cC_g^1([a,b],\bF)$ the norm $\|\cdot \|_1$, where $\|f\|_1:=\max\{\|f\|_0,\|f'_g\|_0\}$, for every $f\in\cB\cC_g^1([a,b],\bF)$. If $\{f_{n}\}_{n\in\mathbb{N}}\ss\cB\cC_g^1([a,b],\bF)$ converges to $f$ in $\cB\cC_g^1([a,b],\bF)$, then, since $\|\cD f_n-\cD f\|_0=\|(f_n)'_g-f'_g\|_0\leq \|f_n-f\|_1$, we have that $\|\cD f_n-\cD f\|_0\xrightarrow{n \to \infty}0$, so $\cD$ is continuous.
\end{proof}

Applying Proposition~\ref{prop3.2} to function series, we obtain the following corollary.

\begin{cor}\label{cor2.4}
	If $\sum_{\n}f_{n}$  is a converging series in $\cB\cC_g^1([a,b],\bF)$ , then $\(\sum_{\n}f_{n}\)'_g=\sum_{\n}(f_{n})'_g$, where the convergence is uniform.
\end{cor}

Now we provide  sufficient conditions under which a countable collection of solutions $\{u_{n}\}_{n\in\mathbb{N}}$ to \eqref{calor} can be summed to obtain another solution of the same equation.
\begin{thm}\label{dif_serie}
	Let $\{\lambda_{n}\}_{n\in\mathbb{N}}\ss\bF$, and for each $n\in\mathbb{N}$, let $u_{n}$ be the solution of \eqref{calor} defined by \[u_{n}(t,x)=w_{n}(t)v_{n}(x),\] where $w_{n}$ is a solution of \eqref{p1} and $v_{n}$ solves \eqref{p2}, with $\lambda=\lambda_{n}$. Assume the following conditions hold:
	\begin{enumerate}
		\item For each fixed $x\in [0,L]$, the series $\sum_{n=1}^{\infty}v_{n}(x)w_{n}$ converges in $\cB\cC^1_h([a,b],\bF)$.
		\item For each fixed $t\in [0,T]$, the series $\sum_{n=1}^{\infty}w_{n}(t)v_{n}$ converges in $\cB\cC^2_g([a,b],\bF)$.
	\end{enumerate}
	Then, the series $u:=\sum_{n=1}^{\infty}u_{n}(t,x)$ is
	 a solution of \eqref{calor}.
\end{thm}
\begin{proof}
Given that  $w_{n}$ is a solution of \eqref{p1} and $v_{n}$ solves \eqref{p2}, we have that $w_{n}\in\cB\cC^\infty([0,T],\bF)$ and $v_{n}\in\cB\cC^\infty([0,L],\bF)$	for $\n$. Hence, we can consider the series	$\sum_{n=1}^{\infty}v_{n}(x)w_{n}(\cdot)$ with $x\in[0,L]$ and, using hypothesis 1, apply Corollary~\ref{cor2.4} to differentiate it under the sum sign, getting
\[\frac{\pd}{\pd_gt}\left(\sum_{n=1}^{\infty}v_{n}(x)w_{n}(t)\)=\sum_{n=1}^{\infty}v_{n}(x)(w_{n})'_g(t)=\l c^2\sum_{n=1}^{\infty}v_{n}(x)w_{n}(t),\quad t\in[0,L].\]
We can do the same to $\sum_{n=1}^{\infty}v_{n}(\cdot)w_{n}(t)$ and for $t\in[0,T]$ and differentiate twice the series, arriving to
\[\frac{\pd^2}{\pd_hx^2}\left(\sum_{n=1}^{\infty}v_{n}(x)w_{n}(t)\right)(t,x)=\sum_{n=1}^{\infty}\pd^{2}_{h} (v_{n})''_h(x)w_{n}(t)(t_{0},x)=\l\sum_{n=1}^{\infty} v_{n}(x)w_{n}(t)(t_{0},x),\quad x\in[0,L].\]
It is now clear that $u$ is a solution of \eqref{calor}.
	\end{proof}
\begin{rem}We can make the conditions in the previous result more explicit by substituting them by
	\begin{enumerate}
					\item For each fixed $x_{0}\in (0,L)$, the series $\sum_{n=1}^{\infty}v_{n}(x_{0})w_{n}(t)$ converges pointwise for all $t\in [0,T]$.
				\item For every $x_{0}\in (0,L)$, the series $\sum_{n=1}^{\infty}v_{n}(x_{0})\lambda_{n}w_{n}(t)$ converges uniformly on $[0,T]$.
				\item For each fixed $t_{0}\in (0,T)$, the series $\sum_{n=1}^{\infty}v_{n}(x)w_{n}(t_{0})$ and $\sum_{n=1}^{\infty}(v_{n})'_{h}(x)w_{n}(t_{0})$ converge pointwise for all $x\in [0,L]$.
				\item For each $t_{0}\in (0,T)$, the series $\sum_{n=1}^{\infty}(v_{n})'_{h}(x)\lambda_{n}w_{n}(t_{0})$ and $\sum_{n=1}^{\infty}\lambda_{n}v_{n}(x)\lambda_{n}w_{n}(t_{0})$ converge uniformly on $[0,L]$.
		\end{enumerate}
	\end{rem}

\section{The initial value problem}\label{IVP}

Now our aim is to analyze the following problem:
\begin{equation}\label{pvi_calor}
	\begin{dcases}
		\pd_{g}u(t,x)-c^{2}\pd_{h}^{2}u(t,x)=0, \\
		u(0,x)=u_{0}(x),
	\end{dcases}
\end{equation}
where $c>0$ is a real constant, and $u_{0}$ is a given complex\--valued function defined on $[0,L]$. We assume that the variables $(t,x)$ take values in $[0,T]\times[0,L]$, for some $T,L\in\mathbb{R}^{+}$.
In Section~\ref{solgeneral}, for each $\lambda\in\mathbb{F}$, we constructed a function $u(t,x)\equiv u_{\lambda}(t,x)=w_{\lambda}(t)v_{\lambda}(x)$ that solves \eqref{calor}, where $w_{\lambda}$ solves \eqref{p1} and $v_{\lambda}$ solves \eqref{p2}. \par
Assume that $u_{0}:[0,L]\rightarrow\mathbb{C}$ is a function of the form
\begin{equation}\label{u0}
	u_{0}(x)=a_{0}+b_{0}h(x)+\sum\limits_{n=1}^{m}a_{n}\exp_{h}(\sqrt{\lambda_{n}};0,x)+\sum\limits_{n=1}^{m}b_{n}\exp_{h}(-\sqrt{\lambda_{n}};0,x),
\end{equation}
where $m\in\mathbb{N}$, $\{a_{n}\}_{n=0}^{m}$ and $\{b_{n}\}_{n=0}^{m}$ are finite collections of complex numbers, and $\{\lambda_{n}\}_{n=1}^{m}$ is a finite collection of nonzero complex numbers.\par Throughout the rest of this subsection, we will use the following notation:
\[u_{0}(x)=a_{0}v_{0}(x)+b_{0}\widetilde{v}_{0}(x)+\sum\limits_{n=1}^{m}a_{n}v_{n}(x)+\sum\limits_{n=1}^{m}b_{n}\widetilde{v}_{n}(x),\]
where
\begin{itemize}
	\item $v_{0}$, defined by $v_{0}(x)=1$, is the unique $h$\--absolutely continuous solution of
	\begin{equation*}\label{v0}
		\begin{dcases}
			v''_{h}(x)=0, & x\in[0,L], \\
			v(0)=1, & \\
			v'_{h}(0)=0;
		\end{dcases}
	\end{equation*}
	\item $\widetilde{v}_{0}$, defined by $\widetilde{v}_{0}(x)=h(x)$, is the unique $h$\--absolutely continuous solution of
	\begin{equation*}\label{v0tilde}
		\begin{dcases}
			v''_{h}(x)=0, & x\in[0,L], \\
			v(0)=h(0), & \\
			v'_{h}(0)=1;
		\end{dcases}
	\end{equation*}
	\item For each $n\in\{1,2,\ldots,m\}$, $v_{n}$, defined by $v_{n}(x)=\exp_{h}(\sqrt{\lambda_{n}};0,x)$, is the unique $h$\--absolutely continuous solution of \begin{equation*}\label{vn}
		\begin{dcases}
			v''_{h}(x)=\lambda_{n}v(x), & x\in[0,L], \\
			v(0)=1, & \\
			v'_{h}(0)=\sqrt{\lambda_{n}};
		\end{dcases}
	\end{equation*}
	\item For each $n\in\{1,2,\ldots,m\}$, $\widetilde{v}_{n}$, defined by $\widetilde{v}_{n}(x)=\exp_{h}(-\sqrt{\lambda_{n}};0,x)$, is the unique $h$\--absolutely continuous solution of
	\begin{equation*}\label{vntilde}
		\begin{dcases}
			v''_{h}(x)=\lambda_{n}v(x), & x\in[0,L], \\
			v(0)=1, & \\
			v'_{h}(0)=-\sqrt{\lambda_{n}}.
		\end{dcases}
	\end{equation*}
\end{itemize}

We will also write:
\begin{itemize}
	\item $w_{0}$ is the unique $g$\--absolutely continuous function of problem
	\begin{equation*}\label{w0}
		\begin{dcases}
			w'_{g}(t)=0, & t\in[0,T], \\
			w(0)=a_{0}; &
		\end{dcases}
	\end{equation*}
	\item $\widetilde{w}_{0}$ is the unique $g$\--absolutely continuous solution of the problem
	\begin{equation*}\label{w0tilde}
		\begin{dcases}
			w'_{g}(t)=0, & t\in[0,T], \\
			w(0)=b_{0}; &
		\end{dcases}
	\end{equation*}
	\item For each $n\in\mathbb{N}$, $w_{n}\in\mathcal{AC}_{g}([0,T],\mathbb{R})$ solves
	\begin{equation*}\label{wn}
		\begin{dcases}
			w'_{g}(t)=\lambda_{n} c^{2}w(t), & t\in[0,T], \\
			w(0)=a_{n}; &
		\end{dcases}
	\end{equation*}
	\item For each $n\in\mathbb{N}$, $\widetilde{w}_{n}$ is the only $g$\--absolutely continuous solution of
	\begin{equation*}\label{wntilde}
		\begin{dcases}
			w'_{g}(t)=\lambda_{n} c^{2}w(t), & t\in[0,T], \\
			w(0)=b_{n}. &
		\end{dcases}
	\end{equation*}
\end{itemize}

The following result establishes the existence of a solution of \eqref{pvi_calor} when the initial condition $u_{0}$ is of the form \eqref{u0}.

\begin{thm}\label{thm_sol_pvi}
	Let $u_{0}:[0,L]\rightarrow\mathbb{C}$ be given by \eqref{u0}. Then, the problem \eqref{pvi_calor} admits a solution given by
	\begin{equation*}\label{sol_pvi}
		u(t,x)=w_{0}(t)v_{0}(x)+\widetilde{w}_{0}(t)\widetilde{v}_{0}(x)+\sum\limits_{n=1}^{m}w_{n}(t)v_{n}(x)+\sum\limits_{n=1}^{m}\widetilde{w}_{n}(t)\widetilde{v}_{n}(x).
	\end{equation*}
\end{thm}

\begin{proof}
	From Theorem~\ref{gen_sol}, it follows that the function $u(t,x)$ solves the partial differential equation \eqref{calor}.\par Evaluating $u$ at $(0,x)$ for $x\in[0,L]$, we obtain
	\begin{equation*}
		u(0,x)=a_{0}v_{0}(x)+b_{0}\widetilde{v}_{0}(x)+\sum\limits_{n=1}^{m}a_{n}v_{n}(x)+\sum\limits_{n=1}^{m}b_{n}\widetilde{v}_{n}(x),
	\end{equation*}
	where we have used the initial values:
$w(0)=a_{0}$, $\widetilde{w}(0)=b_{0}
$, $w_{n}(0)=a_{n}$, $\widetilde{w}_{n}(0)=b_{n}$ for $n\in\{1,2,\ldots,m\}$.
	Then, $u(0,x)=u_{0}(x)$, for all $x\in [0,L]$.
\end{proof}

The next step is to investigate whether problem \eqref{pvi_calor} admits a solution for more general initial conditions. We look for a space such that, if the function $u_{0}$ belongs to the closure of this space, we can guarantee the existence of solution of the problem \eqref{pvi_calor}.

We consider $u_{0}$ given by the series:
\begin{equation*}
	u_{0}(x)=a_{0}+b_{0}h(x)+\sum\limits_{n=1}^{\infty}a_{n}\exp_{h}(\sqrt{\lambda_{n}};0,x)+\sum\limits_{n=1}^{\infty}b_{n}\exp_{h}(-\sqrt{\lambda_{n}};0,x),
\end{equation*}
where $\{a_{n}\}_{n=0}^{\infty}$ and $\{b_{n}\}_{n=0}^{\infty}$ are sequences of complex numbers, and $\{\lambda_{n}\}_{n=1}^{\infty}$ is a sequence of nonzero complex numbers. Equivalently, we can write:
\begin{equation}\label{u0_inf}
	u_{0}(x)=a_{0}v_{0}(x)+b_{0}\widetilde{v}_{0}(x)+\sum\limits_{n=1}^{\infty}a_{n}v_{n}(x)+\sum\limits_{n=1}^{\infty}b_{n}\widetilde{v}_{n}(x).
\end{equation}
\begin{thm}
	Assume the series \eqref{u0_inf} converges pointwise, and consider the problem \eqref{pvi_calor} with $u_{0}$ given by \eqref{u0_inf}. Define the sequences of functions $\{u_{n}\}_{n=0}^{\infty}$ and $\{\widetilde{u}_{n}\}_{n=0}^{\infty}$ by \[u_{n}(t,x):=w_{n}(t)v_{n}(x), \quad \widetilde{u}_{n}(t,x):=\widetilde{w}_{n}(t)\widetilde{v}_{n}(x), \text{ for all } (t,x)\in[0,T]\times[0,L].\] Assume these sequences satisfy the conditions of Theorem~\ref{dif_serie}. Then, the function
	\[u(t,x)=\sum\limits_{n=0}^{\infty}u_{n}(t,x)+\sum\limits_{n=0}^{\infty}\widetilde{u}_{n}(t,x),\]
	is well\--defined and is a solution of problem \eqref{pvi_calor}, with $u_{0}$ given by \eqref{u0_inf}.
\end{thm}

\begin{proof}
	By Theorem~\ref{dif_serie}, the function $u$ defined in the statement is a solution of \eqref{calor}.\par
	Moreover, writing $u$ as:
	\begin{equation*}\label{serie_formal}
		u(t,x)=w_{0}(t)v_{0}(x)+\widetilde{w}_{0}(t)\widetilde{v}_{0}(x)+\sum\limits_{n=1}^{\infty}w_{n}(t)v_{n}(x)+\sum\limits_{n=1}^{\infty}\widetilde{w}_{n}(t)\widetilde{v}_{n}(x),
	\end{equation*}
	and setting
 $w(0)=a_{0}$, $\widetilde{w}(0)=b_{0}$, $w_{n}(0)=a_{n}$, $ \widetilde{w}_{n}(0)=b_{n}$, $n\in\mathbb{N}$,
	we conclude, as in Theorem~\ref{thm_sol_pvi}, that $u(0,x)=u_{0}(x)$ for all $x\in[0,L]$.
\end{proof}

\begin{exa}
	Let us consider the derivators $g,h:\mathbb{R}\rightarrow \mathbb{R}$ defined by
		\[g(t)=
	\begin{dcases}
		t, & t\leq \frac{1}{2}, \\
		t+1, & t>\frac{1}{2},
	\end{dcases} \qquad h(x)=
	\begin{dcases}
	x+2, & x\leq1, \\
	3, & 1<x\leq\frac{3}{2}, \\
	2x+1, & x>\frac{3}{2}.
	\end{dcases}\]
and represented in Figure~\ref{fig1}. We also consider the initial value problem
given by
\begin{equation}\label{example}
	\begin{dcases}
		\pd_{g}u(t,x)-\frac{1}{4}\pd_{h}^{2}u(t,x)=0, & (t,x)\in[0,1]\times[0,2],\\
		u(0,x)=1-h(x)+2\exp_{h}\(\sqrt{\frac{3}{5}};0,x\), & x\in[0,2].
	\end{dcases}
\end{equation}
Then, the function
\[u(t,x)=1-h(x)+2\exp_{g}\(\frac{3}{20};0,t\)\exp_{h}\(\sqrt{\frac{3}{5}};0,x\), \quad (t,x)\in[0,1]\times[0,2],\]
is a solution of problem \eqref{example}---see Fig.~\ref{fig2}.
\begin{figure}[h]
	\begin{minipage}[t]{0.48\textwidth}
		\centering
		\includegraphics[width=\linewidth]{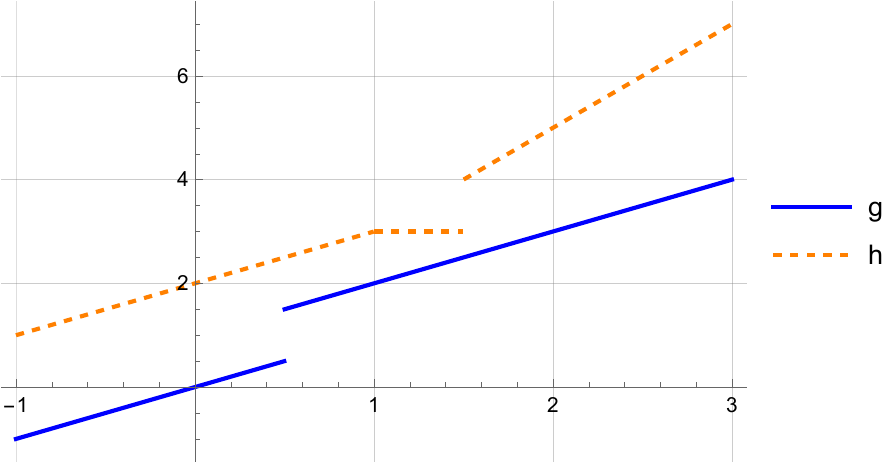}
		\caption{Derivators $g$ and $h$.}\label{fig1}
	\end{minipage}\hfill
	\begin{minipage}[t]{0.48\textwidth}
		\centering
		\includegraphics[width=\linewidth]{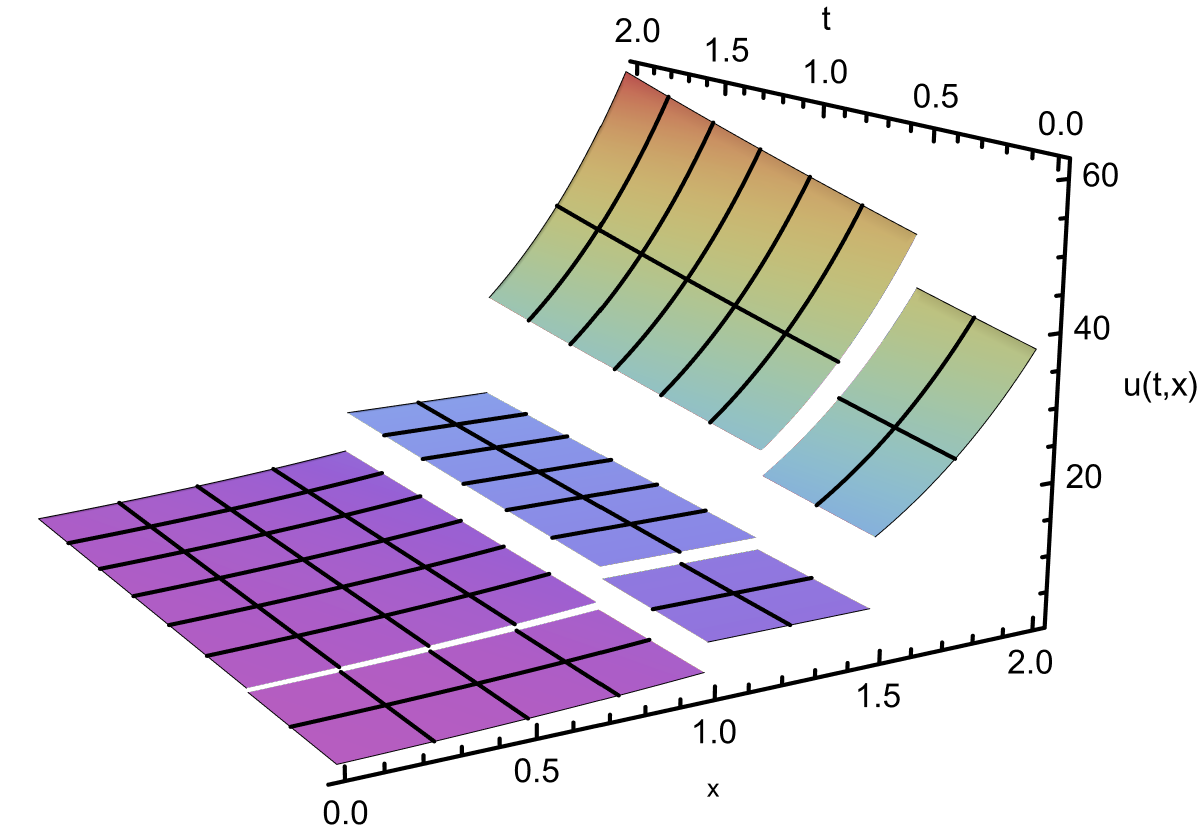}
		\caption{Function $u$, solution of \eqref{example}.}\label{fig2}
	\end{minipage}
\end{figure}
\end{exa}

\section{The heat equation with boundary conditions}\label{Boundary_cond}

We now consider the equation \eqref{calor} with different boundary conditions and seek solutions of these problems.
\subsection{Periodic conditions}
We study the following problem with periodic boundary conditions:
\begin{equation}\label{periodic}
	\begin{dcases}
		\pd_gu(t,x)-c^2\pd_h^2u(t,x)=0,& (t,x)\in [0,T]\times[0,L],\\
		u(t,0)=u(t,L), & t\in [0,T],\\
		\pd_hu(t,0)=\pd_hu(t,L), & t\in [0,T],
\end{dcases} \end{equation}
where $L,T\in\mathbb{R}^{+}$. It will be of interest to consider a particular case of the following second\--order linear problem:
\begin{equation*}
	\begin{dcases}
		v_{g}''(x)+Pv_{g}'(x)+Qv(x)=0, & x\in [0,L] \\
		v(0)=v(L), & \\
		v_{g}'(0)=v'_{g}(L),
	\end{dcases}
\end{equation*}
where $P,Q$ are real numbers. This problem is studied in \cite{periodic}, where the following result is proved.
\begin{thm}[{\cite[Theorem~3.1]{periodic}}]
	Let $g:\mathbb{R}\rightarrow\mathbb{R}$ be a left\--continuous, nondecreasing function, let $f:[0,T]\rightarrow\mathbb{R}$ be $g$\--integrable, and let $P,Q\in \mathbb{R}$ such that equation
	\[\lambda^{2}+P\lambda+Q=0\]
	has two real roots $\lambda_{1},\lambda_{2}$ that satisfy condition
	\[1+\lambda_{k}\Delta g(t)\neq 0, \quad \text{for all } t\in [0,T], \quad k\in \{1,2\}.\]
	Then, the problem \begin{equation}\label{p22periodic}
		\begin{dcases}
			v_{g}''(t)+Pv_{g}'(t)+Qv(x)=f(t), & g-a. e. \ t\in [0,T] \\
			v(0)=v(L), & \\
			v_{g}'(0)=v'_{g}(L),
		\end{dcases}
	\end{equation}
	admits a unique solution in $\mathcal{BC}_{g}^{1}([0,T],\mathbb{R})$, which has an absolutely continuous derivative.
\end{thm}

\begin{rem}
	Although the above Theorem assumes the roots of the characteristic equation $\lambda_{1}$ and $\lambda_{2}$ are real, we will see ---by analyzing the proof of this result in our specific case--- that the same conclusion holds when $\lambda_{1},\lambda_{2}\in\mathbb{C}$.
\end{rem}

As before, we use the method of separation of variables, looking for a solution $u$ of \eqref{periodic} of the form $u(t,x)=w(t)v(x)$.
\begin{thm}
	Assume there exists $\l\in\bF$ such that $\exp_{g}(-\sqrt{\l};0,L)=1$. Then,  $u(t,x)=w(t)v(x)$, where $w$ is given by $w(t)=\exp_{g}(\l c^2;0,t)$, $t\in[0,T]$,  and $v$ solves
    \begin{equation}\label{p2periodic}
		\begin{dcases}
			v_{h}''(x)-\lambda v(x)=0, & h\text{\--a.e. } x\in [0,L] \\
			v(0)=v(L), & \\
			v_{h}'(0)=v'_{h}(L),
		\end{dcases}
	\end{equation}
	is a nontrivial solution in $\cC^2_g([0,T]\times[0,L],\bC)$  of problem \eqref{periodic}.
\end{thm}

\begin{proof}
	By Theorem~\ref{gen_sol}, the function $u(t,x)=w(t)v(x)$ is a nontrivial solution of \eqref{calor} since, for the given $\l\in\bF$, $w$ and $v$ are nontrivial solutions of \eqref{p1} and \eqref{p2}, respectively.

	To satisfy the periodic boundary conditions
	\begin{equation*}
		\begin{dcases}
			u(0,t)=u(L,t), & t\in [0,T],\\
			\pd_hu(0,t)=\pd_hu(0,t), &
	\end{dcases} \end{equation*}
	we require that
	\begin{equation*}
		\begin{dcases}
			v(0)w(t)=v(L)w(t), & t\in [0,T],\\
			v'_{h}(0)w(t)=v'_{h}(L)w(t). &
	\end{dcases} \end{equation*}
	and since $w$ is nontrivial, this implies that $v(0)=v(L)$ and $v'_{h}(0)=v'_{h}(L)$.

Therefore, $v$ is a solution of problem \eqref{p2periodic}, which is a particular case of \eqref{p22periodic} with $P=0$, $Q=-\l$ and $f\equiv 0$.

Following the proof of \cite[Theorem~3.1]{periodic}, we decompose \eqref{p2periodic} into two first\--order problems.
	First, we consider the problem
	\begin{equation}\label{pp_1}
		\begin{dcases}
			u_{h}'(x)+\sqrt\lambda u(x)=0, & h-a. e. \ x\in [0,L], \\
			u(0)=u(L).
		\end{dcases}
	\end{equation}

Take $u_1(x)=\exp_g(-\sqrt{\l};0,x),$ $x\in[0,L]$. By Theorem \ref{orden1} and using $\exp_g(-\sqrt{\l};0,L)=1$, $u_1$ is a solution to problem \eqref{pp_1}.

	 Now we consider the problem
	\begin{equation}\label{pp_2}
		\begin{dcases}
			u_{h}'(x)-\sqrt\lambda u(x)=u_{1}(x), & h-a.e. \ x\in [0,L], \\
			u(0)=u(L).
		\end{dcases}
	\end{equation}
	Let $v$ be its unique $h$\--absolutely continuous solution. Since $v$ solves \eqref{pp_2}, we have that $v(0)=v(L)$. Moreover, $v'_{h}(x)=\sqrt{\lambda}v(x)+u_{1}(x).$
	Differentiating with respect to $h$ both sides of the equality, we get
	\[v''_{h}(x)=\sqrt\lambda v'_{h}(x)+(u_{1})'_{h}(x)=\sqrt\lambda(\sqrt{\lambda}v(x)+u_{1}(x))+(u_{1})'_{h}(x).\]
	Using that $u_{1}$ solves \eqref{pp_1}, we obtain
	\[v''_{h}(x)=\lambda v(x)+\sqrt{\lambda}u_{1}(x)-\sqrt{\lambda}u_{1}(x)=\lambda v(x).\]
	Finally, since $u_{1}$ is a solution of \eqref{pp_1} and $v$ solves \eqref{pp_2}, we have that
	\[v'_{h}(0)=\sqrt{\lambda}v(0)+u_{1}(0)=\sqrt{\lambda}v(L)+u_{1}(L)=v'_{h}(L)\]
	Therefore, $v$ is the unique $h$\--absolutely continuous solution of \eqref{p2periodic}. 	Consequently, the function $u(t,x)=w(t)v(x)$ is a solution of separated variables of \eqref{periodic}.
\end{proof}

\begin{rem}
	We impose the condition $\exp_g(-\sqrt{\l};0,L)=1$ so that problem \eqref{pp_1} admits a nontrivial solution. Without this hypothesis, we would not be able to affirm that function $v$ is nontrivial.
\end{rem}

\subsection{Periodicity of the derivators}

We begin this part by revisiting the classical heat equation:
\begin{equation}\label{calor_clasica}
	u_{t}(t,x)-c^{2}u_{xx}(t,x), \quad (t,x)\in[0,T]\times [0,L],
\end{equation}
for some $T,L\in \mathbb{R}^{+}$, which corresponds to a particular case of \eqref{calor} when $g=h=\operatorname{Id}$.\par
When working with Dirichlet boundary conditions, functions of the form
\begin{equation*}
	u_{n}(t,x)=w_{n}(t)\sin\left(\frac{n\pi}{L}x\right),
\end{equation*}
naturally arise, where $n\in\mathbb{N}$, and $w_{n}$ satisfies the equation
\[w'(t)=-\frac{n^2\pi^{2}}{L^2}c^{2}w(t).\]
These functions $u_{n}$ satisfy both equation \eqref{calor_clasica} and the boundary conditions
\[u(t,0)=u(t,L)=0, \quad t\in[0,T].\]
Similarly, for Neumann boundary conditions,
\[u_{x}(t,0)=u_{x}(t,L)=0, \quad t\in[0,T],\]
functions of the form
\begin{equation*}
	u_{n}(t,x)=w_{n}(t)\cos\left(\frac{n\pi}{L}x\right),
\end{equation*}
where $w_{n}$ satisfies,
\[w'(t)=-\frac{n^2\pi^{2}}{L^2}c^{2}w(t).\]
also appear.

These results rely on fundamental properties of the sine and cosine functions, such as their $2\pi$\--periodicity and symmetry. However, when we consider a general derivator $g:\mathbb{R}\rightarrow\mathbb{R}$, these properties may no longer hold. In particular, without imposing restrictions on $g$, it is not guaranteed that \[\sin_{g}\left(\frac{n\pi}{L};0,L\right)=0.\]

In this part, we aim to determine whether certain conditions on the derivator $g$ allow us to recover key properties of sine and cosine functions for their analogues in Stieltjes differential calculus.\par

Fix $\lambda\in \mathbb{R}^{-}$, and consider the problem
\begin{equation}\label{sistema}
	\begin{dcases}
		\begin{pmatrix} x \\ y \end{pmatrix}'_g(t) = \begin{pmatrix} 0 & \sqrt{-\lambda} \\ -\sqrt{-\lambda} & 0 \end{pmatrix} \begin{pmatrix} x \\ y \end{pmatrix}(t), \; g-a. e.\hspace{0.1cm} \, t \in [0,T), \\
		x(0)=0, \; y(0)=1,
	\end{dcases}
\end{equation}
whose unique solution is the pair $(\sin_g(\sqrt{-\lambda};0,t),\cos_g(\sqrt{-\lambda};0,t))$. The system
\begin{equation*}
\begin{pmatrix} x \\ y \end{pmatrix}'_g(t) = \begin{pmatrix} 0 & \sqrt{-\lambda} \\ -\sqrt{-\lambda} & 0 \end{pmatrix} \begin{pmatrix} x \\ y \end{pmatrix}(t),
\end{equation*}
is associated with the second\--order equation 	\eqref{p2}, which, for $\lambda<0$, admits the general solution
\[v(x)=c_{1}\sin_{g}(\sqrt{-\lambda};0,x)+c_{2}\cos_{g}(\sqrt{-\lambda};0,x)\]
for arbitrary constants $c_{1},c_{2}\in\mathbb{R}$.

Now, let $g:\mathbb{R}\rightarrow\mathbb{R}$ be a left\--continuous, nondecreasing function and assume that it satisfies the following periodicity condition:
\begin{equation}\label{g_condicion}
	g(x+L)-g(y+L)=g(x)-g(y), \text{ for all } x,y\in\mathbb{R}.
\end{equation}
Linear functions of the form $g(x)=ax+b$, where $a\geq 0$ and $b\in \mathbb{R}$, are examples of functions satisfying condition \eqref{g_condicion}.

Moreover, if $g$ satisfies \eqref{g_condicion}, then for every $k\in\mathbb{Z}$, it follows that
\[	g(x+kL)-g(y+kL)=g(x)-g(y), \text{ for all } x,y\in\mathbb{R}.\]

We now prove that, under condition \eqref{g_condicion}, there is a kind of invariance of the $g$\--measure under translations of length $kL$, for all $k\in\mathbb{Z}$.

\begin{pro}
	Let $g$ be a derivator satisfying condition \eqref{g_condicion} and let $A\in\mathcal{M}_{g}$. Then, for any $k\in \mathbb{Z}$, we have that $\mu_{g}(A)=\mu_{g}(A+kL)$,
	where $A+kL:=\left\{x+kL\in\mathbb{R}\ : \ x\in A\right\}$.
\end{pro}

\begin{proof}
	Since $A\in\mathcal{M}_{g}$, we have
	$\mu_{g}(A)=\mu_{g}^{*}(A)$. By the definition of the outer measure associated with $g$ given in \eqref{outer_m}, we have
	\[\mu_{g}^{*}(A)=\inf\left\{\sum_{n=1}^{\infty}\mu_{g}([a_{n},b_{n}))\ : \ A\subset\bigcup_{n=1}^{\infty}[a_{n},b_{n}) \right\}=\inf\left\{\sum_{n=1}^{\infty}(g(b_{n})-g(a_{n}))\ : \ A\subset\bigcup_{n=1}^{\infty}[a_{n},b_{n}) \right\}.\]
	For each $n\in\bN$, condition \eqref{g_condicion} implies that $g(b_{n})-g(a_{n})=g(b_{n}+kL)-g(a_{n}+kL)$.
	Moreover, $A\subset \bigcup_{n=1}^{\infty}[a_{n},b_{n})$ if and only if $A+kL\subset \bigcup_{n=1}^{\infty}[a_{n}+kL,b_{n}+kL)$.

	Therefore,
	\[\mu_{g}^{*}(A)=\inf\left\{\sum_{n=1}^{\infty}(g(b_{n}+kL)-g(a_{n}+kL))\ : \ A+kL\subset\bigcup_{n=1}^{\infty}[a_{n}+kL,b_{n}+kL) \right\}=\mu_{g}^{*}(A+kL).\]

	To conclude the proof, we need to show that, thanks to property \eqref{g_condicion}, $A+kL\in\mathcal{M}_{g}$, that is, $A+kL$ is $g$\--measurable.\par
	Since $A\in \mathcal{M}_{g}$, we know that, for every $E\subset\mathbb{R}$,	$\mu_{g}^{*}(E)=\mu_{g}^{*}(A\cap E)+\mu_{g}^{*}(E\backslash A)$.
	Given $E\subset\mathbb{R}$, we aim to show
	\begin{equation}\label{mu_star}
		\mu_{g}^{*}(E)=\mu_{g}^{*}(\left(A+kL\right)\cap E)+\mu_{g}^{*}(E\backslash \left(A+kL\right)).
	\end{equation}
	Note that
	\[(A+kL)\cap E=A\cap (E-kL)+kL, \quad E\backslash (A+kL)=(E-kL)\backslash A+kL.\]
	Hence, \eqref{mu_star} holds if and only if
	\begin{equation}\label{mu_star2}
		\mu_{g}^{*}(E)=\mu_{g}^{*}(A\cap (E-kL)+kL)+\mu_{g}^{*}((E-kL)\backslash A+kL).
	\end{equation}
	By translation invariance of the outer measure under $kL$, we have
	\begin{align*}\mu_{g}^{*}(E)= & \mu_{g}^{*}(E-kL),\\
	\mu_{g}^{*}(A\cap (E-kL)+kL)= & \mu_{g}^{*}(A\cap (E-kL)),\\\mu_{g}^{*}((E-kL)\backslash A+kL)= &\mu_{g}^{*}((E-kL)\backslash A).\end{align*}
	Substituting into \eqref{mu_star2}, we get
	\[\mu_{g}^{*}(E-kL)=\mu_{g}^{*}(A\cap (E-kL))+\mu_{g}^{*}((E-kL)\backslash A),\]
	which holds since $A\in\mathcal{M}_{g}$.
	Therefore, $A+kL\in\mathcal{M}_{g}$ and since $\mu_{g}(A+kL)=\mu^{*}_{g}(A+kL)=\mu_{g}^{*}(A)=\mu_{g}(A)$, we conclude that $\mu_{g}(A+kL)=\mu_{g}(A)$.
\end{proof}

Next, we aim to show that if $g$ satisfies the condition \eqref{g_condicion}, then the sets $C_{g}$ and $D_{g}$ are preserved under translations of length $L$. This idea is formalized in the following proposition.\par

\begin{pro}
	Assume that $g$ is a derivator satisfying condition \eqref{g_condicion}, and fix $k\in\mathbb{Z}$. If $x\in C_{g}$, then $x+kL\in C_{g}$. Similarly, if $x\in D_{g}$, then $x+kL\in D_{g}$.
\end{pro}
\begin{proof}
	Let $x\in C_{g}$ and $k\in\mathbb{Z}$. We aim to show that $x+kL\in C_{g}$. Since $x\in C_{g}$, there exists $\varepsilon>0$ such that $g$ is constant on the interval $(x-\varepsilon,x+\varepsilon)$. That is, for every $y\in(x-\varepsilon,x+\varepsilon)$, we have $g(y)-g(x)=0$. By condition \eqref{g_condicion}, it follows that for every $z\in(x+kL-\varepsilon,x+kL+\varepsilon)$, \[g(z)-g(x+kL)=g(y)-g(x)=0,\]
	where $y:=z-kL\in(x-\varepsilon,x+\varepsilon)$. Therefore, $g$ is also constant on $(x+kL-\varepsilon, x+kL+\varepsilon)$ and $x+kL\in C_{g}$.\par

	To prove the second claim, assume $x\in D_{g}$ and $k$ is a given integer. Then, by definition, $\Delta g(x)=g(x^{+})-g(x)>0$. Observe that
	\[0<g(x^{+})-g(x)=\lim\limits_{y\rightarrow x^+}(g(y)-g(x))=\lim\limits_{y\rightarrow x^+}(g(y+kL)-g(x+kL))=\Delta g(x+kL),\]
	where the second equality follows from \eqref{g_condicion}. Therefore, $\Delta g(x+kL)>0$, so $x+kL\in D_{g}$.
\end{proof}

\begin{rem}\label{Ltstar}
Recalling expression \eqref{tstar}, we get that $(t+kL)^{*}=t^{*}+kL$. To verify this, first note that if $t\notin C_{g}$, then $t+kL\notin C_{g}$, and the equality is immediate. If $t\in C_{g}$, then $t+kL\in C_{g}$ as shown above. Moreover, if $t\in (a_{n},b_{n})$, where $C_{g}$ is as in \eqref{Cg}, then $t\in (a_{n}+kL,b_{n}+kL)$, and thus $(t+kL)^{*}=b_{n}+kL=t^{*}+kL$.
\end{rem}

\begin{thm}\label{cond_period}
	Let $(x(t),y(t))$ be the unique solution of \eqref{sistema}. If the derivator $g$ satisfies \eqref{g_condicion}, then, for every $k\in \mathbb{Z}$, the pair $(\widetilde{x}(t),\widetilde{y}(t)):=(x(t+kL),y(t+kL))$ also satisfies the system
	\begin{equation*}
	\begin{pmatrix} x \\ y \end{pmatrix}'_g(t) = \begin{pmatrix} 0 & \sqrt{-\lambda} \\ -\sqrt{-\lambda} & 0 \end{pmatrix} \begin{pmatrix} x \\ y \end{pmatrix}(t).
	\end{equation*}
\end{thm}
\begin{proof}
	Let $t_{0}\in [0,T]\backslash \left(C_{g}\cup D_{g}\right)$ and assume $x'_{g}(t_{0}+kL)$ and $y'_{g}(t_{0}+kL)$ exist. Then:
	\[\widetilde{x}'_{g}(t_{0})=\lim\limits_{t\rightarrow t_{0}}\frac{\widetilde{x}(t)-\widetilde{x}(t_{0})}{g(t)-g(t_{0})}=\lim\limits_{t\rightarrow t_{0}}\frac{x(t+kL)-x(t_{0}+kL)}{g(t)-g(t_{0})}=\lim\limits_{t\rightarrow t_{0}}\frac{x(t+kL)-x(t_{0}+kL)}{g(t+kL)-g(t_{0}+kL)}.\]
	If we now use the change of variables $s=t+kL$, we obtain that
	\[\lim\limits_{t\rightarrow t_{0}}\frac{x(t+kL)-x(t_{0}+kL)}{g(t+kL)-g(t_{0}+kL)}=\lim\limits_{s\rightarrow t_{0}+kL}\frac{x(s)-x(t_{0}+kL)}{g(s)-g(t_{0}+kL)}=x'_{g}(t_{0}+kL).\]
	Since $(x(t),y(t))$ solves \eqref{sistema},
	\[x'_{g}(t_{0}+kL)=\sqrt{-\lambda} y(t_{0}+kL)=\sqrt{-\lambda} \widetilde{y}(t_{0}),\]
	Hence, $\widetilde{x}'_{g}(t_{0})=\sqrt{-\lambda}\widetilde{y}(t_{0}).$\par
	If $t_{0}\in C_{g}\cup D_{g}$, then:
	\[\widetilde{x}'_{g}(t_{0})=\lim\limits_{t\rightarrow t_{0}^{+}}\frac{\widetilde{x}(t)-\widetilde{x}(t_{0}^{*})}{g(t)-g(t_{0}^{*})}=\lim\limits_{t\rightarrow t_{0}^+}\frac{x(t+kL)-x((t_{0}+kL)^{*})}{g(t+kL)-g((t_{0}+kL)^{*})}=x'_{g}(t_{0}+kL)=\sqrt{-\lambda}y(t_{0}+kL).\]
	Hence, $\widetilde{x}'_{g}(t_{0})=\sqrt{-\lambda}{y}(t_{0}+kL)=\sqrt{-\lambda}\widetilde{y}(t_{0})$.\par
	The computation for $\widetilde{y}'_{g}(t_{0})$ is analogous, yielding $\widetilde{y}'_{g}(t_{0})=-\sqrt{-\lambda}\widetilde{x}(t_{0})$, $t\in[0,T)$.
\end{proof}

Continuing with the notation from Theorem~\ref{cond_period}, note that if $\widetilde{x}(0)=0$ and $\widetilde{y}(0)=1$, then $(\widetilde{x},\widetilde{y})$ would also be a solution of \eqref{sistema}. By uniqueness of solution of this initial value problem, it would follow that $(\widetilde{x},\widetilde{y})=(x,y)$, meaning that the functions $x$ and $y$ ---that is, $\sin_{g}(\sqrt{-\lambda};0,\cdot)$ and $\cos_{g}(\sqrt{-\lambda};0,\cdot)$, respectively--- are $kL$\--periodic.\par
However, Theorem~\ref{cond_period} only guarantees
$\widetilde{x}(0)=x(kL)$ and $\widetilde{y}(0)=y(kL)$,
and, a priori, we have no information about the values $x(kL)$ and $y(kL)$. To address this, we turn to the series expansion of $\sin_{g}$ and $\cos_{g}$, given in \eqref{serie_sen} and \eqref{serie_cos}, respectively.

We now seek sufficient conditions under which the following hold:
\begin{equation}\label{sin_0_L}
	\sin_{g}(\sqrt{-\lambda};0,0)=\sin_{g}(\sqrt{-\lambda};0,L)=0,
\end{equation}
\begin{equation}\label{cos_0_L}
	\cos_{g}(\sqrt{-\lambda};0,0)=\cos_{g}(\sqrt{-\lambda};0,L)=1.
\end{equation}
The proofs follow directly from the series representations \eqref{serie_sen} and \eqref{serie_cos} and are therefore omitted.

\begin{pro}
	Let $g$ be a derivator such that
	\begin{equation}\label{serie_sen_L}
		\sum\limits_{n=0}^{\infty}(-1)^{n}\left(\sqrt{-\lambda}\right)^{2n+1}\frac{g_{2n+1}(L)}{(2n+1)!}=0.
	\end{equation}
	Then, \eqref{sin_0_L} holds.
\end{pro}

\begin{pro}
	Let $g$ be a derivator such that \begin{equation*}\label{serie_cos_L}
		\sum\limits_{n=0}^{\infty}(-1)^{n}\left(\sqrt{-\lambda}\right)^{2n}\frac{g_{2n}(L)}{(2n)!}=1.
	\end{equation*}
	Then, \eqref{cos_0_L} holds.
\end{pro}

Combining these last results with those obtained in Section~\ref{solgeneral}, we state the following theorem.

\begin{thm}
	Let $g, h:\mathbb{R}\rightarrow\mathbb{R}$ be to left\--continuous, nondecreasing functions, $\lambda\in\mathbb{R}^{-}$, and let $L, T>0$. Suppose $g$ satisfies condition \eqref{serie_sen_L}. Then, for any $a\in\mathbb{R}$, the function
	\[u(t,x)=a\exp_{g}(\lambda c^{2};0,t)\sin_{h}(\sqrt{-\lambda};0,x)\] solves the problem
	\begin{equation*}
		\begin{dcases}
			\pd_{g}u(t,x)-c^{2}\pd_{h}^{2}u(t,x)=0, & (t,x)\in(0,T)\times(0,L),\\
			u(t,0)=u(t,L)=0.
		\end{dcases}
	\end{equation*}
	Under the same assumptions, for any $b\in \mathbb{R}$, the function
	\[u(t,x)=b\exp_{g}(\lambda c^{2};0,t)\cos_{h}(\sqrt{-\lambda};0,x),\]
	is a solution of the problem
	\begin{equation*}
		\begin{dcases}
			\pd_{g}u(t,x)-c^{2}\pd_{h}^{2}u(t,x)=0, & (t,x)\in(0,T)\times(0,L),\\
			u'_{h}(t,0)=u'_{h}(t,L)=0.
		\end{dcases}
	\end{equation*}
\end{thm}

\section{General solution of the heat equation with Stieltjes derivatives and a derivator of two variables}\label{Two_derivators}

In Section~\ref{solgeneral}, we considered two derivators: one associated with the variable $t$ and another associated with the variable $x$. In this section, we study the heat equation driven by a single derivator depending on two variables. More precisely, we consider a function $G:\mathbb{R}^{2}\to\mathbb{R}$ such that, for each fixed $t\in\mathbb{R}$ and $x\in\mathbb{R}$, the functions $G(\cdot,x)$ and $G(t,\cdot)$ are derivators.

The particular case $G(t,x)=g(t)+h(x)$ ---see Section~\ref{ssd}--- extends in a straightforward manner the framework of product time scales (cf.~\cite[Chapter~6]{bohner2016multivariable}). However, more general choices of $G$ offer greater flexibility, allowing for intricate interactions between the measures associated with $t$ and $x$, which in turn influence the structure of the resulting differential equations.

 Next, we define the notion of derivator of a finite number of variables.
\begin{dfn}\label{def_nderiv}
	Given $n\in\mathbb{N}$, we say that a function $G:\mathbb{R}^{n}\rightarrow\mathbb{R}$ is a \textit{derivator of $n$ variables} (or an \textit{$n$\--variable derivator}) if, for each $i\in\{1,2,\ldots,n\}$, and for fixed values $x_{1},\ldots,x_{i-1},x_{i+1},\ldots,x_{n}$, the real\--valued function defined by \[x\mapsto G(x_{1},x_{2},\ldots,x_{i-1},x,x_{i+1},\ldots,x_{n}),\]
	is nondecreasing and left\--continuous.
\end{dfn}

We now assume that $G:\mathbb{R}^{2}\rightarrow\mathbb{R}$ is a derivator of two variables, and we proceed to define the derivatives of a function of two variables with respect to $G$.\par

\begin{rem}\label{notation_G}
	We introduce some convenient notation. For a fixed $t_{0}\in\mathbb{R}$, we define the function
	\[x\in\mathbb{R}\mapsto h_{t_{0}}(x):=G(t_{0},x)\in\mathbb{R}. \]
	Similarly, for a fixed $x_{0}\in\mathbb{R}$, we define
	\[t\in\mathbb{R}\mapsto g_{x_{0}}(t):=G(t,x_{0})\in\mathbb{R}.\]
\end{rem}

\begin{dfn}
	Let $I,J\subset\mathbb{R}$ be intervals, and let $f:I\times J\rightarrow\mathbb{F}$ be a function. Given $(t,x)\in I\times J$, we define the \textit{$G$\--derivative of $f$ with respect to $t$ at $(t,x)$} as the following limit, provided it exists:
	\begin{equation*}
		\frac{\pd f}{\pd_{G} t}(t,x)=\begin{dcases}
			\lim\limits_{s\rightarrow t}\frac{f(s,x)-f(t,x)}{G(s,x)-G(t,x)}, & t\notin D_{g_{x}}\cup C_{g_{x}},\\
			\lim\limits_{s\rightarrow t^{+}}\frac{f(s,x)-f(t^{*},x)}{G(s,x)-G(t^{*},x)}, & t\in D_{g_{x}}\cup C_{g_{x}},
		\end{dcases}
	\end{equation*}
	where we have adapted the notation \eqref{tstar} to the derivator $G(\cdot,x)=g_{x}(\cdot)$.\par
	Analogously, the \textit{$G$\--derivative of $f$ with respect to $x$ at $(t,x)$} is defined as
	\begin{equation*}
		\frac{\pd f}{\pd_{G} x}(t,x)=\begin{dcases}
			\lim\limits_{y\rightarrow x}\frac{f(t,y)-f(t,x)}{G(t,y)-G(t,x)}, & t\notin D_{h_{t}}\cup C_{h_{t}},\\
			\lim\limits_{y\rightarrow x^{+}}\frac{f(t,y)-f(t,x^{*})}{G(t,y)-G(t,x^{*})}, & t\in D_{h_{t}}\cup C_{h_{t}},
		\end{dcases}
	\end{equation*}
	where $x^{*}$ is defined analogously to \eqref{tstar}, but for the derivator $G(t,\cdot)=h_{t}(\cdot)$, and provided the corresponding limit exist.
\end{dfn}
\begin{rem}\label{remintd}
	To understand this type of derivative, it is important to consider the case of function $G$ that has nonzero partial derivatives. In that case,
	\[\frac{\pd u}{\pd_{G}t}(t,x)=	\lim\limits_{y\rightarrow x}\frac{u(t,y)-u(t,x)}{G(t,y)-G(t,x)}=	\lim\limits_{y\rightarrow x}\frac{\frac{u(t,y)-u(t,x)}{y-x}}{\frac{G(t,y)-G(t,x)}{y-x}}=\frac{\frac{\pd u}{\pd t}(t,x)}{\frac{\pd G}{\pd t}(t,x)},\]
	and, analogously,
	\[\frac{\pd u}{\pd_{G}x}(t,x)=\frac{\frac{\pd u}{\pd x}(t,x)}{\frac{\pd G}{\pd x}(t,x)}.\]

	We may wonder now which $G$ returns the usual partial derivatives $\bR^2$. It is enough to impose the equations
	\[\frac{\pd u}{\pd t}(t,x)=\frac{\pd u}{\pd_{G}t}(t,x)=\frac{\frac{\pd u}{\pd t}(t,x)}{\frac{\pd G}{\pd t}(t,x)},\quad \frac{\pd u}{\pd x}(t,x)=\frac{\pd u}{\pd_{G}x}(t,x)=\frac{\frac{\pd u}{\pd x}(t,x)}{\frac{\pd G}{\pd x}(t,x)},\]
	from which we get
	\[\frac{\pd G}{\pd t}(t,x)=1,\quad \frac{\pd G}{\pd x}(t,x)=1,\]
	that is, $G(t,x)=t+x+C$ for some constant $C$.
\end{rem}

Having introduced the concepts of a derivator of two variables and the derivatives with respect to these function, we now present the heat equation associated with a derivator of two variables $G$:
\begin{equation}\label{Gcalor}
	\frac{\pd u}{\pd_{G}t}(t,x)-c^{2}\frac{\pd^2 u}{\pd_{G}x^2}(t,x)=0,
\end{equation}
where $c\in\mathbb{R}$ is a positive constant.

\subsection{Particular case: sum of derivators of one variable}\label{ssd}

We start by considering the kind of function $G$ that arised in Remark~\ref{remintd}, that is,
\begin{equation}\label{sum_deriv}
	G(t,x)=g(t)+h(x),
\end{equation}
for some functions $g,h:\mathbb{R}\rightarrow\mathbb{R}$.
The function $G$ defined by \eqref{sum_deriv} satisfies Definition~\ref{def_nderiv} if and only if $g$ and $h$ are single\--variable derivators.\par

We begin by examining the $G$\--derivative of a function with respect to $t$.

\begin{pro}\label{Gderiv_sum}
	Let $I,J\subset\mathbb{R}$ be intervals and fix $(t,x)\in I\times J$. Let $u:I\times J\rightarrow\mathbb{F}$ be a function, and let
	$G$ be a derivator of two variables of the form \eqref{sum_deriv}. Then the $G$\--derivative of $u$ with respect to $t$ at $(t,x)$ exists if and only if the $g$\--derivative $\pd_{g}u(t,x)=u'_{g}(\cdot,x)|_{t}$ exists. Moreover,
	\begin{equation*}
		\frac{\pd u}{\pd_{G}t}(t,x)=\pd_{g}u(t,x).
	\end{equation*}
\end{pro}
\begin{proof}
	We distinguish two cases:
	\begin{enumerate}
		\item If $t\in C_{g}\cup D_{g}$, then
		\begin{equation*}
			\frac{\pd u}{\pd_{G}t}(t,x)=\lim\limits_{s\rightarrow t^+}\frac{u(s,x)-u(t,x)}{g(s)+h(x)-\left(g(t)+h(x)\right)}=\lim\limits_{s\rightarrow t^+}\frac{u(s,x)-u(t,x)}{g(s)-g(t)}=\pd_{g}u(t,x).
		\end{equation*}

		\item If $t\notin C_{g}\cup D_{g}$, then
		\begin{equation*}
			\frac{\pd u}{\pd_{G}t}(t,x)=\lim\limits_{s\rightarrow t}\frac{u(s,x)-u(t^{*},x)}{g(s)+h(x)-\left(g(t^{*})+h(x)\right)}=\lim\limits_{s\rightarrow t}\frac{u(s,x)-u(t^{*},x)}{g(s)-g(t^{*})}=\pd_{g}u(t,x).
		\end{equation*}
	\end{enumerate}
	Thus, $\pd u/\pd_{G}t$ exists at $(t,x)$ if and only if $\pd_{g}u(t,x)$ exists, and when they exist, they are equal.
\end{proof}

We now present the analogous result for the $G$\--derivative with respect to $x$. The proof is omitted, as it mirrors that of Proposition~\ref{Gderiv_sum}.

\begin{pro}
	Let $I,J\subset\mathbb{R}$ intervals, and fix $(t,x)\in I\times J$. Let $u:I\times J\rightarrow\mathbb{F}$ be a function, and let $G$ be a derivator of two variables of the form \eqref{sum_deriv}. Then the $G$\--derivative of $u$ with respect to $x$ at $(t,x)$ exists if and only if the $h$\--derivative $\pd_{h}u(t,x)=u'_{h}(t,\cdot)|_{x}$ exists. Moreover,
	\begin{equation*}
		\frac{\pd u}{\pd_{G}x}(t,x)=\pd_{h}u(t,x).
	\end{equation*}
\end{pro}

We conclude that a function of the form $u(t,x)=w(t)v(x)$ is a solution of \eqref{Gcalor} if and only if it satisfies
\begin{equation*}
	\pd_gu(t,x)-c^2\pd_h^2u(t,x)=0.
\end{equation*}

That is, $u$ solves \eqref{Gcalor} if and only if it satisfies equation \eqref{calor}. Therefore, the results from Section~\ref{solgeneral} apply.

\subsubsection{$G$\--polynomials}
In \cite{widder1975heat}, the \textit{heat polynomial} of degree $n$ is defined by
\begin{equation}\label{def_heatpol}
	v_n(t,x):=\sum_{k=0}^{\lfloor n / 2\rfloor} \frac{n!}{(n-2k)!k!} t^kx^{n-2k}.
\end{equation}
These functions are used to construct a series of the form
\[u(t,x)=\sum_{n=0}^{\infty}a_nv_n(t,x)=\sum_{n=0}^{\infty}a_n\sum_{k=0}^{\lfloor n / 2\rfloor} \frac{n!}{(n-2k)!k!} t^kx^{n-2k},\]
which solves the classical heat equation (with diffusion constant $c=1$):
\[\frac{\pd u}{\pd t}-\frac{\pd^2u}{\pd x^2}=0.\]
In this representation, $u$ is expressed as a series of two\--variable polynomials $t^mx^n$, for $m,n\in\{0,1,2,\ldots\}$. In this section,
we aim to construct polynomials $G_{m,n}(t,x)$ that serve as counterparts to the monomials $t^mx^n$. Using them, we seek a solution to equation \eqref{Gcalor} of the form
\[u(t,x)=\sum_{m=0}^{\infty}\sum_{n=0}^{\infty}a_{m,n}G_{m,n}(t,x).\]
To this end, we follow the ideas in \cite{analytic} and
consider the operators
\[
\cI H(t,x)=\int_{0}^t H(s,x)\,\operatorname{d}\mu_{g_x}(s),\quad
\cK H(t,x)=\int_{0}^x H(t,y)\,\operatorname{d}\mu_{h_t}(y),
\]
defined for suitably integrable functions $H$.

For the polynomials $G_{m,n}(t,x)$ to be well defined, we need to impose certain restrictions on the derivator $G$. From now on, we restrict ourselves to those derivators of the form \eqref{sum_deriv}. In this case, we have
\[
\cI H(t,x)=\int_{0}^t H(s,x)\,\operatorname{d}\mu_{g}(s),\quad
\cK H(t,x)=\int_{0}^x H(t,y)\,\operatorname{d}\mu_{h}(y).
\]
First, we prove that these operators commute when applied to functions of separated variables.
\begin{pro}\label{int_sepvar}
	Let $G(t,x)=g(t)+h(x)$ be a derivator. If $H(t,x)=H_1(t)H_2(x)$, $(t,x)\in [0,T]\times[0,L]$ is a function of separated variables, then
	\[\cI\cK H(t,x)=\cK\cI H(t,x), \quad(t,x)\in[0,T]\times[0,L].\]
\end{pro}
\begin{proof}
	We have
	\begin{align*}
		\cI\cK H(t,x)&=\int_{0}^t \cK H(s,x)\,\operatorname{d}\mu_{g}(s)=\int_{0}^t \int_{0}^x H(s,y)\,\operatorname{d}\mu_{h}(y)\,\operatorname{d}\mu_{g}(s)\\
		&=\left(\int_{0}^{t}H_1(s)\, \operatorname{d}\mu_{g}(s)\right)\left(\int_{0}^{x}H_2(y)\, \operatorname{d}\mu_{h}(y)\right),
	\end{align*}
	and, analogously,
	\begin{align*}
		\cK\cI H(t,x)&=\int_{0}^x \cI H(t,y)\,\operatorname{d}\mu_{h}(y)=\int_{0}^x \int_{0}^t H(s,y)\,\operatorname{d}\mu_{g}(s)\,\operatorname{d}\mu_{h}(y)\\
		&=\left(\int_{0}^{t}H_1(s)\, \operatorname{d}\mu_{g}(s)\right)\left(\int_{0}^{x}H_2(y)\, \operatorname{d}\mu_{h}(y)\right),
	\end{align*}
	from which the result is immediate.
\end{proof}

\begin{dfn}
	Let $G$ be a derivator of the form \eqref{Gderiv_sum}. We define $G_{0,0}(x)=1$ for every $(t,x)\in [0,T]\times[0,L]$. For each $m,n\in\mathbb{N}$, we define $G_{m,n}:[0,T]\times[0,L]\rightarrow\mathbb{R}$ recursively by the equations
	\[  G_{m,0}(t,x)=
	m\int_{0}^tG_{m-1,0}(s,x)\operatorname{d}\mu_{g}(s),\qquad  G_{m,n}(t,x)=
	n\int_{0}^xG_{m,n-1}(t,y)\operatorname{d}\mu_{h}(y).\]
\end{dfn}

\begin{rem}
	From the definition of the operators $\cI$ and $\cK$, it follows directly that
	\[G_{m,0}(t,x)=m\,\cI G_{m-1,0}(t,x), \quad G_{m,n}(t,x)=n\,\cK G_{m,n-1}(t,x).\]
\end{rem}

\begin{rem}
	As a consequence of Theorem \ref{TFC2}, we have that
	\begin{equation*}
		\frac{\pd G_{m,n}}{\pd_{G}t}(t,x)=mG_{m-1,n}(t,x),
	\end{equation*}
	and
	\begin{equation*}
		\frac{\pd G_{m,n}}{\pd_{G}x}(t,x)=nG_{m,n-1}(t,x).
	\end{equation*}
\end{rem}

If $g$ and $h$ are derivators such that $g(0)=h(0)=0$, then we obtain an explicit expression for the $G$\--polynomials in terms of the $g$\--monomials $g_{0,m}\equiv g_m$ and the $h$\--monomials $h_{0,n}\equiv h_n$, introduced in Definition~\ref{g_monomials}.

\begin{pro}\label{form_Gpol}
	Let $g,h$ be derivators such that $g(0)=h(0)=0$. Then, for $m,n\in \mathbb{N}\cup \{0\}$,
	\[G_{m,n}(t,x)=g_m(t)h_n(x), \quad (t,x)\in[0,T]\times[0,L].\]
	\begin{proof}
		First, we prove by induction on $m$ that
		\begin{equation}\label{form_Gm0}
			G_{m,0}(t,x)=g_m(t), \quad (t,x)\in[0,T]\times[0,L].
		\end{equation}
		For $m=0$, the result is immediate. The case $m=1$ follows from $g(0)=0$:
		\begin{equation*}
			G_{1,0}(t,x)=\int_{0}^{t}1\operatorname{d}\mu_{g}(s)=g(t)-g(0)=g(t)=g_1(t).
		\end{equation*}
		Assume that \eqref{form_Gm0} holds for some  $m\in\mathbb{N}$. Then, 		\begin{align*}
			G_{m+1,0}(t,x)&=(m+1)\int_{0}^tG_{m,0}(s,x)\operatorname{d}\mu_{g}(s)=(m+1)\int_{0}^tg_m(s)\operatorname{d}\mu_{g}(s) \\
			&=(m+1)\frac{g_{m+1}(t)}{m+1}=g_{m+1}(t).
		\end{align*}
		Now, for a fixed $m\in\bN\cup\{0\}$, we prove \eqref{form_Gpol} by induction on $n$. For $n=1$,
		\[		G_{m,1}(t,x)=\int_{0}^{x}G_{m,0}(t,y)\operatorname{d}\mu_{h}(y)=\int_{0}^{x}g_m(t)\operatorname{d}\mu_{h}(y)=g_m(t)\left(h(x)-h(0)\right)=g_m(t)h(x).\]
		Assume that \eqref{form_Gpol} holds for some $n\in\bN$. Then,
		\begin{align*}
			G_{m,n+1}(t,x)&=(n+1)\int_{0}^xG_{m,n}(t,y)\operatorname{d}\mu_{h}(y)=(n+1)\int_{0}^xg_m(t)h_n(y)\operatorname{d}\mu_{h}(y) \\
			&=(n+1)g_m(t)\frac{h_{n+1}(x)}{n+1}=g_m(t)h_{n+1}(x).
		\end{align*}
	\end{proof}
\end{pro}

If $g$ is a continuous derivator, then it follows from \cite[Proposition 3.15]{analytic} that $g_n(t)=g(t)^n$, $t\in\bR$. This leads us to the following result.
\begin{cor}\label{form_Gpol_cont}
	Let $g,h$ be continuous derivators such that $g(0)=h(0)=0$. Then, for $m,n\in \mathbb{N}\cup \{0\}$, we have that $G_{m,n}(t,x)=g(t)^mh(x)^n$ for every $(t,x)\in[0,T]\times[0,L]$.
\end{cor}

\begin{rem}
	Taking $g$ and $h$ to be the identity functions  in Corollary~\ref{form_Gpol_cont}  yields the classical two\--variable monomials $t^mx^n$.
\end{rem}

We now seek a solution to \eqref{Gcalor} of the form
\begin{equation}\label{series_sol}
	u(t,x)=\sum\limits_{m=0}^\infty\sum\limits_{n=0}^\infty a_{m,n}G_{m,n}(t,x)=\sum\limits_{m=0}^\infty\sum\limits_{n=0}^\infty a_{m,n}g_m(t)h_n(x),
\end{equation}
assuming that the series converges. We rely on the following result.
\begin{thm}\label{a_m,n}
	Let $u$ be a function of the form \eqref{series_sol}, where the series converges uniformly on $[0,T]\times[0,L]$. Assume that the series can be differentiated term by term. Then $u$ is a solution to \eqref{Gcalor} if and only if the coefficients of the series satisfy the relations
	\begin{equation}\label{form_coef}
		a_{m+1, n}=\frac{c^2(n+2)(n+1)}{(m+1)} a_{m, n+2}
		, \quad n,m\in\{0,1,2,\ldots\},
	\end{equation}
	in which case we can write $u$ as
	\begin{equation}\label{u_series_expression}
		u(t,x)=\sum\limits_{n=0}^\infty a_{0,n}\sum\limits_{k=0}^{\lfloor n/2\rfloor}  \frac{c^{2k}n!}{(n-2m)!k!}g_k(t)h_{n-2k}(x).
	\end{equation}

\end{thm}
\begin{proof}
	First, we show the coefficients must satisfy \eqref{form_coef}. We have that
	\begin{equation*}
		\frac{\pd u(t,x)}{\pd_G t}=\sum_{m=1}^{\infty}\sum_{n=0}^{\infty}a_{m,n}mG_{m-1,n}(t,x)=\sum_{m=0}^{\infty}\sum_{n=0}^{\infty}a_{m+1,n}(m+1)G_{m,n}(t,x),
	\end{equation*}
	and
	\begin{equation*}
		\frac{\pd^2 u(t,x)}{\pd_G t^2}=\sum_{m=0}^{\infty}\sum_{n=2}^{\infty}a_{m,n}n(n-1)G_{m,n-2}(t,x)=\sum_{m=0}^{\infty}\sum_{n=0}^{\infty}a_{m,n+2}(n+2)(n+1)G_{m,n}(t,x).
	\end{equation*}
	Thus, $u(t,x)$ is a solution to \eqref{Gcalor} if and only if for $(t,x)\in [0,T]\times[0,L],$
	\begin{equation*}
		\sum_{m=0}^{\infty}\sum_{n=0}^{\infty}a_{m+1,n}(m+1)G_{m,n}(t,x)-c^2\sum_{m=0}^{\infty}\sum_{n=0}^{\infty}a_{m,n+2}(n+2)(n+1)G_{m,n}(t,x)=0,
	\end{equation*}
	that is,
	\begin{equation}\label{serie_coef}
		\sum_{m=0}^{\infty}\sum_{n=0}^{\infty}\left(a_{m+1,n}(m+1)-c^2a_{m,n+2}(n+2)(n+1)\right)G_{m,n}(t,x)=0.
	\end{equation}
	A sufficient condition for \ref{serie_coef} to hold is
	\begin{equation*}
		a_{m+1,n}(m+1)-c^2 a_{m,n+2}(n+2)(n+1)=0, \quad n,m\in\{0,1,2,\ldots\},
	\end{equation*}
	which is equivalent to \eqref{form_coef}.

	To prove the converse, assume that $u$ is a solution to \eqref{Gcalor} of the form \eqref{series_sol}. Our goal is to show that \eqref{form_coef} holds, which is equivalent to $\lambda_{m,n}=0$, where
	\[\lambda_{m,n}:=a_{m+1,n}(m+1)-c^2 a_{m,n+2}(n+2)(n+1).\]
	Using this notation, we can rewrite \eqref{serie_coef} as
	\begin{equation*}
		\sum_{m=0}^\infty\sum_{n=0}^{\infty}\l_{m,n}G_{m,n}(t,x)=0.
	\end{equation*}
	If $u$ solves \eqref{Gcalor}, then the coefficients satisfy \eqref{serie_coef}. Let $k,l\in\bN\cup\{0\}$. Differentiating both sides of \eqref{serie_coef} $k$ times with respect to $g$ and $l$ with respect to $h$, we obtain
	\[0=\frac{\pd ^k}{\pd_G t^k}\frac{\pd^l}{\pd_{G}x^l}\left(	\sum_{m=0}^{\infty}\sum_{n=0}^{\infty}\l_{m,n}G_{m,n}(t,x)\right), \quad(t,x)\in[0,T]\times[0,L],\]
	that is, for $(t,x)\in[0,T]\times[0,L]$,
	\begin{equation}\label{deriv_tk_lx}
		0=	\sum_{m=k}^{\infty}\sum_{n=l}^{\infty}\lambda_{m,n}m(m-1)\ldots(m-k+1) n(n-1)\ldots(n-l+1)G_{m-k,n-l}(t,x).
	\end{equation}
	Taking into account that  $G_{m,n}(0,0)=g_m(0)h_n(0)=0$ unless  $m=n=0$, evaluating \eqref{deriv_tk_lx} at $(t,x)=(0,0)$ yields
	$0=\l_{k,l}$.
	Since $k,l\in\bN\cup\{0\}$ are arbitrary, the first part of the result follows.

	To obtain \eqref{u_series_expression}, we use induction in \eqref{form_coef} to arrive at
	\[a_{m,n}=\frac{c^{2m}(n+2m)\ldots(n+1)}{m!}a_{0,n+2m}=\frac{c^{2m}(n+2m)!}{n!m!}a_{0,n+2m}, \quad n,m\in \{0,1,\ldots\}.\]

	Hence,
	\[u(t,x)=\sum\limits_{m=0}^\infty\sum\limits_{n=0}^\infty \frac{c^{2m}(n+2m)!}{n!m!}a_{0,n+2m}G_{m,n}(t,x).\]
	We can now rewrite $u$, using the index $k=n+2m$ and then relabeling indices, as
	\[u(t,x)=\sum\limits_{k=0}^\infty a_{0,k}\sum\limits_{m=0}^{\lfloor k/2\rfloor}  \frac{c^{2m}k!}{(k-2m)!m!}G_{m,k-2m}(t,x)=\sum\limits_{n=0}^\infty a_{0,n}\sum\limits_{k=0}^{\lfloor n/2\rfloor}  \frac{c^{2k}n!}{(n-2m)!k!}g_k(t)h_{n-2k}(x).\qedhere\]
\end{proof}

\begin{rem}
	A function of the form \eqref{series_sol} satisfies
	\[u(0,x)=\sum_{m=0}^{\infty}\sum_{n=0}^{\infty}a_{m,n}G_{m,n}(0,x)=	\sum_{n=0}^{\infty}a_{0,n}G_{0,n}(0,x)=\sum_{n=0}^{\infty}a_{0,n}h_n(x).\]
	Therefore, if we have the initial condition
	\begin{equation*}
		u(0,x)=\sum_{n=0}^{\infty}\alpha_nh_n(x),
	\end{equation*}
	for some $\{\alpha_n\}_{n\geq0}\subset\bC$, then the coefficients
	$a_{0,n}$ for $n\in \{0,1,\ldots\}$ are determined:
	$a_{0,n}=\alpha_n$.
\end{rem}

Following the approach in \cite{widder1975heat}, we name the polynomials that appear in \eqref{u_series_expression}.

\begin{dfn}
	We define the \textit{heat $G$\--polynomial} of degree $n$ as
	\[v^G_n(t,x):=\sum_{k=0}^{\lfloor n / 2\rfloor} \frac{n!c^{2k}}{(n-2 k)!k!} g_k(t)h_{n-2k}(x).\]
\end{dfn}

\begin{rem}
	Observe that, for $(t,x)\in[0,T]\times[0,L]$ and $n\in\bN\cup\{0\}$,
	\[v^G_n(0,x)=h_n(x),\qquad v^G_{2n}(t,0)=\frac{(2n)!(c^2)^n}{n!}g_n(t), \qquad v^G_{2n+1}(t,0)=0.\]
    Moreover, we have formulas similar to those presented in \cite[Chapter I, Section 5]{widder1975heat} for the derivatives of the heat $G$\--polynomials. In particular,
\begin{align*}
	\frac{\pd }{\pd_Gx}v_n^G(t,x)&=\pd_h v_n^G(t,x)=\sum_{k=0}^{\lfloor n / 2\rfloor} \frac{n!c^{2k}}{(n-2 k)!k!} g_k(t)(n-2k)h_{n-2k-1}(x) \\
	 &=n\sum_{k=0}^{\lfloor (n-1) / 2\rfloor} \frac{(n-1)!c^{2k}}{(n-2 k-1)!k!} g_k(t)h_{n-2k-1}(x),
\end{align*}
that is,
\begin{equation}\label{pd_x}
	\frac{\pd }{\pd_Gx}v_n^G(t,x)=nv_{n-1}^G(t,x).
\end{equation}
Analogously, 
\begin{align*}
	\frac{\pd }{\pd_Gt}v_n^G(t,x)&=\pd_g v_n^G(t,x)=\sum_{k=1}^{\lfloor n / 2\rfloor} \frac{n!c^{2k}}{(n-2 k)!k!} kg_{k-1}(t)(n-2k)h_{n-2k-1}(x) \\
	&=c^2n(n-1)\sum_{k=0}^{\lfloor (n-2) / 2\rfloor} \frac{n!c^{2k}}{(n-2 k-2)!k!} g_k(t)h_{n-2k-2}(x).
\end{align*}
Thus,
\begin{equation}\label{pd_t}
	\frac{\pd }{\pd_Gt}v_n^G(t,x)=c^2n(n-1)v_{n-2}^G(t,x).
\end{equation}
Note that, from equations \eqref{pd_x} and \eqref{pd_t} we obtain
\[\frac{\pd }{\pd_Gt}v_n^G(t,x)=c^2\frac{\pd^2 }{\pd_Gx^2}v_n^G(t,x).\]
\end{rem}

We now study sufficient conditions for the convergence of the series
\[u(t,x)=\sum_{n=0}^{\infty}\alpha_nv^G_n(t,x).\]
We will use the following result, which is a direct consequence of \cite[Corollary 3.12]{analytic}.

\begin{pro}\label{bound_gn}
	Let $g$ be a derivator and let $g_{x_0,n}$ be the $g$\--monomial centered at $x_0\in\bR$ with $g(x_0)=0$. Then, for $x\geq 0$, we have that $0\leq g_n(x)\leq g(x)^n$.
\end{pro}

In \cite{widder1975heat}, several results concerning the convergence of series of classical heat polynomials are established. The following theorems, which describe the corresponding regions of convergence, will be essential for the analysis of the convergence of our series of heat $G$\--polynomials. In both cases, we consider a series of the form
\begin{equation}\label{series_heatpol}
	\sum_{n=0}^{\infty} a_{n} v_n(t,x),
\end{equation}
where $v_n(t,x)$ denotes the heat polynomials defined in~\eqref{def_heatpol}.

\begin{thm}[{\cite[Theorem 4.1]{widder1975heat}}]
	Consider the series \eqref{series_heatpol}
	and assume that there exist $t_0>0$ and $x_0\neq 0$ such that
	\begin{equation*}
		\sum_{n=0}^{\infty}a_{n}v_n(t_0,x_0),
	\end{equation*}
	converges. Then, \eqref{series_heatpol} converges absolutely for $|t|<t_0$.
\end{thm}

\begin{thm}[{\cite[Theorem 5.1]{widder1975heat}}] Consider the series \eqref{series_heatpol}
	and define
	\[\frac{1}{\sigma}:=\lim\limits_{n\to \infty}\frac{2n|a_n|^\frac{2}{n}}{e},\]
	where $\sigma=0$ if the limit is $+\infty$, and $\sigma=+\infty$ if this limit equals $0$.

	Then, the series \eqref{series_heatpol} converges absolutely for $|t|<\sigma$ and does not converge everywhere for $|t|>\sigma$.
\end{thm}

Combining these results, we obtain the following convergence theorem for series of heat $G$\--polynomials.

\begin{thm}\label{radius_conv}
	Let $G$ be a derivator of the form \eqref{Gderiv_sum} such that $g(0)=h(0)=0$. Consider a series
	\begin{equation}\label{series_heatGpol}
		u(t,x)=\sum_{n=0}^{\infty}a_nv^G_n(t,x),
	\end{equation}
	and define
\begin{equation}\label{radius_form}
	\frac{1}{\sigma}:=\lim\limits_{n\to \infty}\frac{2n|\alpha_n|^\frac{2}{n}}{e},
\end{equation}
	Then the series \eqref{series_heatGpol} converges absolutely for any $(t,x)\in[0,T]\times[0,L]$ such that $g(t)<\frac{\sigma}{c^2}$.
\end{thm}
\begin{proof}
	Note that, by Proposition~\ref{bound_gn},
	\[v_n^G(t,x)=\sum_{k=0}^{\lfloor n / 2\rfloor} \frac{n!c^{2k}}{(n-2 k)!k!} g_k(t)h_{n-2k}(x)\leq \sum_{k=0}^{\lfloor n / 2\rfloor} \frac{n!c^{2k}}{(n-2 k)!k!} g(t)^k h(x)^{n-2k}=v_n(c^2g(t),h(x)).\]
	If $c^2|g(t)|<\sigma$, then Theorem~\ref{radius_conv} yields the desired convergence.
\end{proof}
\begin{rem}
Since $g$ is left\--continuous, nondecreasing, and satisfies $g(0)=0$, the set
	\[\left\{t\in[0,T]\, :\, |g(t)|<\frac{\sigma}{c^2}\right\}\]
	is either $\{0\}$ if $0\in D_g$ and $\Delta g(0)\geq \sigma/c^2$, the interval $[0,T]$ if $g(T)<\sigma/c^2$, or an interval of the form $[0,d)$ or $[0,d]$, with $d\in(0,T)$.
\end{rem}

\begin{rem}\label{deriv_series_heatGpol}
	Note that the series of $G$\--derivatives with respect to $t$ of the heat $G$\--polynomials is again a series of heat $G$\--polynomials:
	\[\sum_{n=0}^{\infty} \frac{\pd}{\pd_{G}t}\left(\alpha_nv_n^G(t,x)\right)=\sum_{n=2}^{\infty}\alpha_nc^2n(n-1)v_{n-2}^G(t,x)=\sum_{n=0}^{\infty}\alpha_{n+2}c^2(n+2)(n+1)v_n^G(t,x).\]
Moreover, it has the same radius of convergence as \eqref{u_series_heatGpol}. Indeed, if we denote
\[\overline{\alpha_n}:=\alpha_{n+2}c^2(n+2)(n+1),\]
and apply \eqref{radius_form} to these coefficients, we have
\begin{align*}
	\lim\limits_{n\to \infty}\frac{2n|\overline{\alpha_n}|^\frac{2}{n}}{e}&=	\lim\limits_{n\to \infty}\frac{2n|\alpha_{n+2}|^\frac{2}{n}\left(c^2(n+2)(n+1)\right)^\frac{2}{n}}{e}\\
	&=\lim\limits_{n\to \infty}\left(\frac{2(n+2)|\alpha_{n+2}|^\frac{2}{n+2}}{e}\right)^{\frac{n+2}{n}}\frac{2n\left(c^2(n+2)(n+1)\right)^\frac{2}{n}e^\frac{2}{n}}{\left(2(n+2)\right)^\frac{n+2}{n}}=\frac{1}{\sigma}.
\end{align*}
An analogous argument shows that the series of higher\--order $G$-derivatives (with respect to both $t$ and $x$) are also series of heat $G$-polynomials with the same radius of convergence as \eqref{u_series_heatGpol}. 
\end{rem}

Finally, we construct a series of heat $G$\--polynomials that solves \eqref{Gcalor}. The proof follows directly from the previous results.

\begin{thm}
	Let $G$ be a derivator of the form \eqref{Gderiv_sum} such that $g(0)=h(0)=0$, and consider the function
	\begin{equation}\label{u_series_heatGpol}
		u(t,x):=\sum_{n=0}^{\infty}\alpha_n v^G_n(t,x).
	\end{equation}
	If $g(T)<\frac{\sigma}{c^2}$, where
$\sigma$ is given by \eqref{radius_form},
	then $u$
	is a solution to \eqref{Gcalor}.
\end{thm}
\begin{proof}
	By Theorem~\ref{radius_conv}, the series \eqref{u_series_heatGpol} converges absolutely on $[0,T]\times[0,L]$. In particular, the series
	$\sum_{n=0}^{\infty}|\alpha_n|v_n^G(T,L)$
	converges. For $(t,x)\in [0,T]\times[0,L]$, since $g$ and $h$ are nondecreasing, we have
	\[0\leq v_n^G(t,x)=\sum_{k=0}^{\lfloor n / 2\rfloor} \frac{n!c^{2k}}{(n-2 k)!k!} g_k(t)h_{n-2k}(x)\leq \sum_{k=0}^{\lfloor n / 2\rfloor} \frac{n!c^{2k}}{(n-2 k)!k!} g_k(T)h_{n-2k}(L)=v_n^G(T,L).\]
	Thus, by the Weierstrass uniform convergence theorem, the series \eqref{u_series_heatGpol} converges uniformly on $[0,T]\times [0,L]$. Moreover, by Remark~\ref{deriv_series_heatGpol} the series of $G$\--derivatives of $u$ converge absolutely and uniformly on $[0,T]\times[0,L]$. This result, together with Corollary~\ref{cor2.4}, allows us to differentiate the series term by term and apply Theorem~\ref{a_m,n} to conclude that $u$ is a solution of \eqref{Gcalor}. 
\end{proof}

\subsection{Particular case: derivator of separated variables}

We now focus on derivators of the form:
\begin{equation}\label{G_derivador}
	G(t,x)=g(t)h(x).
\end{equation}

\begin{rem}
	In this case, for derivators of the form \eqref{G_derivador} and for fixed $x_{0}\in\mathbb{R}$, the function $G(t,x_{0})=g(t)h(x_{0})$ is left\--continuous and nondecreasing. If $h(x_{0})\neq 0$, then the function $t\mapsto g(t)$ is left\--continuous. Furthermore, if $h(x_{0})>0$, then $g$ is nondecreasing; whereas if $h(x_{0})<0$, then $g$ is nonincreasing. Analogously, for fixed $t_{0}\in\mathbb{R}$, if $g(t_{0})\neq0$, then $h$ is left\--continuous and either nondecreasing (if $g(t_{0})>0$), or nonincreasing (if $g(t_{0})<0$).\par
	Henceforth, we assume that both $g$ and $h$ are usual derivators, that is, left\--continuous, nondecreasing functions. Additionally, we assume that $g(t)>0$, $h(x)>0$ for all $t,x\in\mathbb{R}$.
\end{rem}

Under these assumptions, for every $x\in\mathbb{R}$, we have $C_{g_{x}}=C_{g}$ and $D_{g_{x}}=D_{g}$. Similarly, for all $t\in \mathbb{R}$, we have $C_{h_{t}}=C_{h}$ and $D_{h_{t}}=D_{h}$ .

\begin{rem}
The simplest nontrivial derivator of this kind is $G(t,x)=tx$.
In this case, $G$ measures the signed area of the rectangle with opposite corners $(0,0)$ and $(t,x)$.
In other words, it is the distribution function of the Lebesgue measure on $\bR^2$.
This suggests a connection with the usual notion of differentiation and partial derivatives.
However, as we observed in Remark~\ref{remintd}, the appropriate derivator in this setting should instead be of the form $t+x+C$.
The key observation that reconciles these two facts is given by the following identities.

If we denote
\[
G_t := \frac{\pd G}{\pd t}(t,x),
\qquad
G_x := \frac{\pd G}{\pd x}(t,x),
\]
then
\begin{equation*}
	\frac{\pd u}{\pd t}(t,x)
	=
	\lim_{s\to t}\frac{u(s,x)-u(t,x)}{s-t}
	=
	\lim_{s\to t}
	\frac{u(s,x)-u(t,x)}{\frac{\pd G}{\pd x}(s,x)-\frac{\pd G}{\pd x}(t,x)}
	=
	\frac{\pd u}{\pd_{G_x} t}(t,x),
\end{equation*}
and
\begin{equation*}
	\frac{\pd u}{\pd x}(t,x)
	=
	\lim_{y\to x}\frac{u(t,y)-u(t,x)}{y-x}
	=
	\lim_{y\to x}
	\frac{u(t,y)-u(t,x)}{\frac{\pd G}{\pd t}(t,y)-\frac{\pd G}{\pd t}(t,x)}
	=
	\frac{\pd u}{\pd_{G_t} x}(t,x).
\end{equation*}

That is, the partial derivative of $u$ with respect to $t$ is not the derivative of $u$ with respect to $G$ itself and the variable $t$, but rather the derivative of $u$ with respect to the \emph{variation} of $G$ in the $t$\--direction (that is, $G_x$) and the variable $t$. An analogous statement holds for differentiation with respect to $x$.

This interpretation becomes clearer if we examine the increment of $G$:
\[
\Delta G
= G(t+\Delta t, x+\Delta x) - G(t,x)
= (t+\Delta t)(x+\Delta x) - tx
= t\Delta x + x\Delta t + \Delta t \Delta x,
\]
which resembles the product rule for Stieltjes derivatives \cite{fernandez2025consequences}.

If the increment occurs only in the $t$\--variable (that is, $\Delta x = 0$), then
\[
\Delta G = x \Delta t.
\]
Thus, even when varying only $t$, the increment of $G$ still depends on the value of $x$.
To eliminate this dependence, one should not consider $G$ itself in absolute terms, but rather its variation with respect to $x$. Indeed,
\[
\Delta G_x = \Delta t.
\]

Finally, observe that $G$ can be characterized as the solution of the system
\[
G_t = x,
\qquad
G_x = t,
\]
which explicitly describes its dependence on both variables.
\end{rem}

As in the case of two single\--variable derivators, we look for a solution of \eqref{Gcalor} of the form $u(t,x)=w(t)v(x)$. Under the assumption \eqref{G_derivador}, we can find a simpler expression for the derivatives of $u$ with respect to $G$.

\begin{pro}
	Let $G:\mathbb{R}^{2}\rightarrow\mathbb{R}$ be a derivator of two variables of the form \eqref{G_derivador},
	and let $u:[0,T]\times [0,L]\rightarrow\mathbb{F}$ be a function of two variables of the form $u(t,x)=w(t)v(x)$, for each $(t,x)\in [0,T]\times[0,L]$. Then, for any $(t,x)\in[0,T]\times [0,L]$, the following hold:
	\begin{enumerate}
		\item The $G$\--derivative of $u$ with respect to $t$ at $(t,x)$ exists if and only if $w$ is $g$\--differentiable at $t$. In this case,
		\begin{equation}\label{pd_tG}
			\frac{\pd u}{\pd_{G}t}(t,x)=\frac{v(x)}{h(x)}w'_{g}(t),
		\end{equation}
		\item The second order $G$\--derivative of $u$ with respect to $x$ at $(t,x)$ exists if and only if $v$ is twice $h$\--differentiable at $x$. In this case,
		\begin{equation}\label{pd2_xG}
			\frac{\pd^{2} u}{\pd_{G}x^{2}}(t,x)=\frac{w(t)}{g(t)^{2}}v''_{h}(x).
		\end{equation}
	\end{enumerate}
\end{pro}

\begin{proof}
	Fix $x\in [0,L]$. If $t\notin C_{g}\cup D_{g}$, the existence of the $G$\--derivative of $u$ with respect to $t$ at $(t,x)$ is equivalent to the existence of the limit
	\[\frac{\pd u}{\pd_{G}t}(t,x)=\lim\limits_{s\rightarrow t}\frac{u(s,x)-u(t,x)}{G(s,x)-G(t,x)}=\lim\limits_{s\rightarrow t}\frac{w(s)v(x)-w(t)v(x)}{g(s)h(x)-g(t)h(x)}=\frac{v(x)}{h(x)}\lim\limits_{s\rightarrow t}\frac{w(s)-w(t)}{g(s)-g(t)}. \]
	That is, $\pd u/\pd_{G}t$ exists at $(t,x)$ if and only if $w'_{g}(t)$ exists, and, in that case,
	\begin{equation*}
		\frac{\pd u}{\pd_{G}t}(t,x)=\frac{v(x)}{h(x)}w'_{g}(t).
	\end{equation*}
	If $t\in C_{g}\cup D_{g}$, the existence of the limit
	\[\lim\limits_{s\rightarrow t^{+}}\frac{w(s)v(x)-w(t^{*})v(x)}{g(s)h(x)-g(t^{*})h(x)}\]
	is equivalent to the existence of
	\[\lim\limits_{s\rightarrow t^{+}}\frac{w(s)-w(t^{*})}{g(s)-g(t^{*})}.\]
	So again, the $G$\--derivative of $u$ at $(t,x)$ exists if and only if $w'_{g}(t)$ exists, and we obtain the same formula:
	\begin{equation*}
		\frac{\pd u}{\pd_{G}t}(t,x)=\lim\limits_{s\rightarrow t^{+}}\frac{w(s)v(x)-w(t^{*})v(x)}{g(s)h(x)-g(t^{*})h(x)}=\frac{v(x)}{h(x)}\lim\limits_{s\rightarrow t^{+}}\frac{w(s)-w(t^{*})}{g(s)-g(t^{*})}=\frac{v(x)}{h(x)}w'_{g}(t).
	\end{equation*}

	Analogous reasoning applies to the $G$\--derivative with respect to $x$. Fix $t\in[0,T]$. The $G$\--derivative of $u$ with respect to $x$ at $(t,x)$ exists if and only if $v'_{h}(x)$ exists. If $x\notin C_{h}\cup D_{h}$, then
	\begin{equation*}
		\frac{\pd u}{\pd_{G}x}(t,x)=\lim\limits_{y\rightarrow x}\frac{w(t)v(y)-w(t)v(x)}{g(t)h(y)-g(t)h(x)}=\frac{w(t)}{g(t)}\lim\limits_{y\rightarrow x}\frac{v(y)-v(x)}{h(y)-h(x)}=\frac{w(t)}{g(t)}v'_{h}(x),
	\end{equation*}
	whereas if $x\in C_{h}\cup D_{h}$, we obtain
	\begin{equation*}
		\frac{\pd u}{\pd_{G}x}(t,x)=\lim\limits_{y\rightarrow x^{+}}\frac{w(t)v(y)-w(t)v(x^{*})}{g(t)h(y)-g(t)h(x^{*})}=\frac{w(t)}{g(t)}\lim\limits_{y\rightarrow x^{+}}\frac{v(y)-v(x^{*})}{h(y)-h(x^{*})}=\frac{w(t)}{g(t)}v'_{h}(x).
	\end{equation*}

	Assume now the second\--order $G$\--derivative of $u$ with respect to $x$ at $(t,x)$ exists. Then, the first\--order $G$\--derivative of $u$ with respect to $x$ at $(t,x)$ must also exist, and
	\[\frac{\pd u}{\pd_{G}x}(t,x)=\frac{w(t)}{g(t)}v'_{h}(x).\]
	For $x\notin C_{h}\cup D_{h} $, we have
	\begin{equation*}
		\frac{\pd^{2} u}{\pd_{G}x^{2}}(t,x)=\lim\limits_{y\rightarrow x}\frac{\frac{w(t)}{g(t)}v'_{h}(y)-\frac{w(t)}{g(t)}v'_{h}(x)}{g(t)h(y)-g(t)h(x)}=\frac{w(t)}{g(t)^{2}}\lim\limits_{y\rightarrow x}\frac{v'_{h}(y)-v'_{h}(x)}{h(y)-h(x)}=\frac{w(t)}{g(t)^{2}}v''_{h}(x).
	\end{equation*}
	If, instead, $x\in C_{h}\cup D_{h}$, then
	\begin{equation*}
		\frac{\pd^{2} u}{\pd_{G}x^{2}}(t,x)=\lim\limits_{y\rightarrow x^{+}}\frac{\frac{w(t)}{g(t)}v'_{h}(y)-\frac{w(t)}{g(t)}v'_{h}(x^{*})}{g(t)h(y)-g(t)h(x^{*})}=\frac{w(t)}{g(t)^{2}}\lim\limits_{y\rightarrow x^{+}}\frac{v'_{h}(y)-v'_{h}(x^{*})}{h(y)-h(x^{*})}=\frac{w(t)}{g(t)^{2}}v''_{h}(x).
	\end{equation*}
	Therefore, $v$ is twice differentiable with respect to $h$ at $x$, and we obtain
	\[	\frac{\pd^{2} u}{\pd_{G}x^{2}}(t,x)=\frac{w(t)}{g(t)^{2}}v''_{h}(x).\]
	The converse follows analogously.
\end{proof}

Substituting the expression $u(t,x)=w(t)v(x)$ into equation \eqref{Gcalor}, and using identities \eqref{pd_tG} and \eqref{pd2_xG}, we find that $u$ is a solution of \eqref{Gcalor} if and only if
\begin{equation}\label{asterisco2}
	\frac{v(x)}{h(x)}w'_{g}(t)-c^{2}\frac{w(t)}{g(t)^{2}}v''_{h}(x)=0.
\end{equation}

\begin{thm}
	Let $G$ be a derivator of two variables, and consider equation \eqref{Gcalor}. A nontrivial function $u:[0,T]\times[0,L]\rightarrow\mathbb{F}$ of the form $u(t,x)=w(t)v(x)$ is a solution of \eqref{Gcalor} if and only if there exists $\lambda\in\mathbb{F}$ such that the non trivial functions $v$ and $w$ satisfy
	\begin{equation}\label{problema_v}
		v''_{h}(x)=\frac{\lambda}{h(x)}v(x), \quad x\in [0,L],
	\end{equation}
	\begin{equation}\label{problema_w}
		w'_{g}(t)=\frac{\lambda c^{2}}{g(t)^{2}}w(t), \quad t\in [0,T].
	\end{equation}
\end{thm}

\begin{proof}
	Suppose $u$ is nontrivial. Then there exists $(t_{0},x_{0})\in [0,T]\times[0,L]$ such that $u(t_{0},x_{0})\neq 0$, which implies $w(t_{0})\neq 0$ and $v(x_{0})\neq 0$. Evaluating \eqref{asterisco2} at $(t_{0},x_{0})$, we obtain
	\begin{equation*}
		\frac{v(x_{0})}{h(x_{0})}w'_{g}(t_{0})-c^{2}\frac{w(t_{0})}{g(t_{0})^{2}}v''_{h}(x_{0})=0,
	\end{equation*}
	which is equivalently to
	\begin{equation*}
		\frac{v''_{h}(x_{0})h(x_{0})}{v(x_{0})}=\frac{g(t_{0})^{2}}{c^{2}}\frac{w'_{g}(t_{0})}{w(t_{0})}.
	\end{equation*}
	Define
	\[\lambda:=	\frac{v''_{h}(x_{0})h(x_{0})}{v(x_{0})}=\frac{g(t_{0})^{2}}{c^{2}}\frac{w'_{g}(t_{0})}{w(t_{0})}.\]
	Now, substituting $t=t_{0}$ into \eqref{asterisco2}, we have
	\begin{equation*}
		v''_{h}(x)=\frac{\lambda}{h(x)}v(x), \quad x\in [0,L],
	\end{equation*}
	and evaluating \eqref{asterisco2} at $x=x_{0}$, we get
	\begin{equation*}
		w'_{g}(t)=\frac{\lambda c^{2}}{g(t)^{2}}w(t), \quad t\in [0,T].
	\end{equation*}

	Conversely, suppose there exists $\lambda\in\mathbb{F}$ such that
	problems \eqref{problema_v} and \eqref{problema_w} admit nontrivial solutions $v$ and $w$, respectively. Then their product $u(t,x)=w(t)v(x)$ satisfies
	\[\frac{\pd u}{\pd_{G}t}(t,x)-c^{2}\frac{\pd^{2} u}{\pd_{G}x^{2}}(t,x)=\frac{v(x)}{h(x)}w'_{g}(t)-c^{2}\frac{w(t)}{g(t)^{2}}v''_{h}(x)=\frac{v(x)}{h(x)}\frac{\lambda c^{2}}{g(t)^{2}}w(t)-c^{2}\frac{w(t)}{g(t)^{2}}\frac{\lambda}{h(x)}v(x)=0.\]

	Thus, $u(t,x)=w(t)v(x)$ is a nontrivial solution of \eqref{Gcalor} if and only if there exists $\lambda\in\mathbb{F}$ such that the problems \eqref{problema_v} and \eqref{problema_w} admit the nontrivial solutions $v$ and $w$.
\end{proof}

We now analyze problem \eqref{problema_w}. Define the function $p:[0,T]\rightarrow\mathbb{F}$ by
\[p(t)=\frac{\lambda c^{2}}{g(t)^{2}}.\]
If $p\in \mathcal{L}_{g}^{1}([0,T],\mathbb{R})$, then by Theorem~\ref{orden1}, the general solution of problem \eqref{problema_w} is given by
\begin{equation*}
	w(t)=x_{0}\exp_{g}(p;0,t).
\end{equation*}

To treat problem \eqref{problema_v}, we cannot apply \cite[Theorem~5.3]{Fernandez2022} directly, since the coefficients are not constant. Instead, we refer to the results in \cite{second_non_constant_coef}, where the authors study equations of the form:
\begin{equation}\label{phom2_nonconstant}
	\begin{dcases}
		v_{g}''(t)+P(t)v_{g}'(t)+Q(t)v(t)=f(t), & t\in [0,T], \\
		v(0)=x_{0}, & \\
		v_{g}'(0)=v_{0},
	\end{dcases}
\end{equation}
where $g$ is a derivator, $f,P,Q\in\mathcal{BC}_{g}([0,T],\mathbb{C})$ and $x_{0},v_{0}\in \mathbb{C}$.

In \cite{second_non_constant_coef} it is shown that problem \eqref{phom2_nonconstant} is equivalent to the system
\begin{equation*}
	\begin{dcases}
		x'_{g}(t)=y(t), & t\in[0,T],\\
		y'_{g}(t)=f(t)-P(t)y(t)-Q(t)x(t), & \\
		x(0)=x_{0}, \quad y(0)=v_{0},
	\end{dcases}
\end{equation*}
and, as proved in \cite[Theorem~5.58]{tesis}, this problem admits a unique solution.

Consequently, if $\lambda/h(x)\in\mathcal{BC}_{h}([0,L],\mathbb{C})$, then the problem
\begin{equation}\label{pvi_v}
	\begin{dcases}
		v_{h}''(x)-\frac{\lambda}{h(x)}v(x)=0, & x\in [0,L] \\
		v(0)=x_{0} & \\
		v_{h}'(0)=v_{0},
	\end{dcases}
\end{equation}
admits a unique solution.

In the following lemma, we adapt \cite[Lemma 3.12]{second_non_constant_coef} to our particular setting.
\begin{lem}
	Let $v_{1}, v_{2}\in\mathcal{BC}_{g}^{2}([0,T],\mathbb{C})$ be solutions of \eqref{problema_v} such that
	\begin{equation}\label{cond_indep}
		v_{1}(0)(v_{2})'_{h}(0)-v_{2}(0)(v_{1})'_{h}(0)\neq 0.
	\end{equation}
	Then, the function $v=c_{1}v_{1}+c_{2}v_{2}$, where
	\[c_{1}=\frac{(v_{2})'_{h}(0)x_{0}-v_{0}v_{2}(0)}{v_{1}(0)(v_{2})'_{h}(0)-v_{2}(0)(v_{1})'_{h}(0)},\qquad c_{2}=\frac{v_{0}v_{1}(0)-(v_{1})'_{h}(0)x_{0}}{v_{1}(0)(v_{2})'_{h}(0)-v_{2}(0)(v_{1})'_{h}(0)},\]
	is the unique solution of the initial value problem \eqref{pvi_v}.
\end{lem}

Since \eqref{problema_v} is a second\--order linear homogeneous equation, its set of solutions is a two\--dimensional vector space. Moreover, \cite[Corollary 3.14]{second_non_constant_coef} provides a sufficient condition for two solutions of be linearly independent and, therefore, to form a basis. We apply this result equation \eqref{problema_v}.

\begin{pro}
	Let $\lambda\in\mathbb{R}$ and let $h:\mathbb{R}\rightarrow\mathbb{R}$ be a derivator such that
	\[1-\frac{\lambda}{h(x)}\Delta h(x)^{2}\neq 0, \quad x\in[0,L]\cap D_{h}.\]
	Let $v_{1}, v_{2}\in\mathcal{BC}_{h}^{2}([0,L],\mathbb{C})$ be two solutions of \eqref{problema_v} such that \eqref{cond_indep} holds. Then $v_{1}$ and $v_{2}$ are linearly independent.
\end{pro}

We conclude this section by presenting a solution of \eqref{Gcalor} obtained via separation of variables.

\begin{thm}
	Let $c\in\mathbb{R}^{+}$, and let $G:\mathbb{R}^{2}\rightarrow\mathbb{R}$ be a derivator of two variables of the form $G(t,x)=g(t)h(x)$, where $g$ and $h$ are left\--continuous, nondecreasing functions such that $g(t)>0$ and $h(x)>0$ for all $t,x\in\mathbb{R}$. For each $\lambda\in\mathbb{F}$, define
	\[p_{\lambda}(t)=\frac{\lambda c^{2}}{g(t)^{2}}, \quad t\in\mathbb{R},\quad q_{\lambda}(x)=\frac{\lambda}{h(x)}, \quad x\in\mathbb{R}.\]

	Assume $p_{\lambda}\in\mathcal{L}_{g}^{1}([0,T],\mathbb{F})$ and $q_{\lambda}\in \mathcal{BC}_{h}([0,L],\mathbb{F})$. Then, equation \eqref{Gcalor} admits a solution. Moreover, if $v$ is a solution of \eqref{problema_v}, then, for a fixed $x_{0}\in\mathbb{R}$, the function
	\[u(t,x)=x_{0}\exp_{g}(p;0,t)v(x)\]
	is a solution of \eqref{Gcalor}.
\end{thm}

\section{Conclusions}

In this article, we have focused on obtaining explicit solutions of the heat equation. As a direction for future work, it would be of interest to establish existence and uniqueness results, analogous to those obtained in \cite{Fernandez2020}.

Moreover, using the derivators of several variables introduced in Section \ref{Two_derivators}, a natural step would be to study equation
\eqref{Gcalor} for a general derivator $G(t,x)$.

Nevertheless, in this more general setting, we find some difficulties when attempting to obtain a solution via separation of variables.

Using the notations established in Remark~\ref{notation_G}, the $G$\--derivatives of $u(t,x)=w(t)v(x)$ can be computed as follows:
\begin{equation*}
	\frac{\pd u}{\pd_{t}G}(t,x)=\lim\limits_{s\rightarrow t}\frac{w(s)v(x)-w(t)v(x)}{G(s,x)-G(t,x)}=v(x)\lim\limits_{s\rightarrow t}\frac{w(s)-w(t)}{g_{x}(s)-g_{x}(t)}=v(x)w'_{g_{x}}(t),
\end{equation*}
\begin{equation*}
	\frac{\pd u}{\pd_{x}G}(t,x)=\lim\limits_{y\rightarrow x}\frac{w(t)v(y)-w(t)v(x)}{G(t,y)-G(t,x)}=w(t)\lim\limits_{y\rightarrow x}\frac{v(y)-v(x)}{h_{t}(y)-h_{t}(x)}=w(t)v'_{h_{t}}(x).
\end{equation*}
Hence, if $u(t,x)=w(t)v(x)$ is a solution of \eqref{Gcalor}, then it satisfies
\begin{equation}\label{asterisco3}
	v(x)w'_{g_{x}}(t)-c^{2}w(t)v''_{h_{t}}(x)=0, \quad (t,x)\in(0,T)\times(0,L).
\end{equation}
Suppose $u$ is not the trivial solution. Fix $(t_{0},x_{0})\in[0,T]\times[0,L]$ such that $u(t_{0},x_{0})\neq 0$. Substituting $(t,x)=(t_{0},x_{0})$ into \eqref{asterisco3} gives
\[v(x_{0})w'_{g_{x_{0}}}(t_{0})-c^{2}w(t_{0})v''_{h_{t_{0}}}(x_{0})=0,\]
or, equivalently,
\[\frac{v''_{h_{t_{0}}}(x_{0})}{v(x_{0})}=\frac{w'_{g_{x_{0}}}(t_{0})}{c^{2}w(t_{0})}.\]
We denote this constant by $\lambda$:
\begin{equation*}\label{lambda}
	\lambda:=\frac{v''_{h_{t_{0}}}(x_{0})}{v(x_{0})}=\frac{w'_{g_{x_{0}}}(t_{0})}{c^{2}w(t_{0})}.
\end{equation*}
Now, evaluating \eqref{asterisco3} at $t=t_{0}$, we obtain the equation:
\[v(x)w'_{g_{x}}(t_{0})-c^{2}w(t_{0})v''_{h_{t_{0}}}(x)=0, \quad x\in[0,L],\]
which implies
\begin{equation}\label{aux_v}
	v''_{h_{t_{0}}}(x)=\frac{w'_{g_{x}}(t_{0})}{c^{2}w(t_{0})}v(x), \quad x\in[0,L].
\end{equation}
To express the right\--hand side in terms of the constant $\lambda$, we analyze the relationship between $w'_{g_{x}}(t)$ and $w'_{g_{x_{0}}}(t)$.

\begin{pro}\label{rel_gx_gy}
	Let $x,y\in [0,L]$ be such that $C_{g_{x}}=C_{g_{y}}$ and $D_{g_{x}}=D_{g_{y}}$. Assume that $g_{y}$ is $g_{x}$\--differentiable on $[0,T]$. Then, for any $t\in [0,T]$, if $w$ is a $g_{y}$\--differentiable function at $t$, then, $w$ is also $g_{x}$\--differentiable at $t$ and
	\begin{equation*}
		w'_{g_{x}}(t)=w'_{g_{y}}(t)(g_{y})'_{g_{x}}(t).
	\end{equation*}
\end{pro}

\begin{proof}
	Fix $x,y\in[0,L]$ and take $t\in [0,T]$.\par If $t\notin C_{g_{x}}\cup D_{g_{x}}$, then
	\[w'_{g_{x}}(t)=\lim\limits_{s\rightarrow t}\frac{w(s)-w(t)}{g_{x}(s)-g_{x}(t)}=\lim\limits_{s\rightarrow t}\frac{w(s)-w(t)}{g_{y}(s)-g_{y}(t)}\frac{g_{y}(s)-g_{y}(t)}{g_{x}(s)-g_{x}(t)}=w'_{g_{y}}(t)(g_{y})'_{g_{x}}(t).\]
	If $t\in D_{g_{x}}$, then
	\[w'_{g_{x}}(t)=\lim\limits_{s\rightarrow t^{+}}\frac{w(s)-w(t)}{g_{x}(s)-g_{x}(t)}=\lim\limits_{s\rightarrow t^{+}}\frac{w(s)-w(t)}{g_{y}(s)-g_{y}(t)}\frac{g_{y}(s)-g_{y}(t)}{g_{x}(s)-g_{x}(t)}=w'_{g_{y}}(t)(g_{y})'_{g_{x}}(t).\]
	Finally, if $t\in C_{g_{x}}$, assuming $t\in (a_{n},b_{n})\subset C_{g_{x}}=C_{g_{y}}$, we have
	\[w'_{g_{x}}(t)=\lim\limits_{s\rightarrow b_{n}^{+}}\frac{w(s)-w(b_{n})}{g_{x}(s)-g_{x}(b_{n})}=\lim\limits_{s\rightarrow b_{n}^{+}}\frac{w(s)-w(b_{n})}{g_{y}(s)-g_{y}(b_{n})}\frac{g_{y}(s)-g_{y}(b_{n})}{g_{x}(s)-g_{x}(b_{n})}=w'_{g_{y}}(t)(g_{y})'_{g_{x}}(t).\qedhere\]
\end{proof}

\begin{rem}
	The assumptions $C_{g_{x}}=C_{g_{y}}$ and $D_{g_{x}}=D_{g_{y}}$ are essential to ensure that the limits are well\--defined. The similar behavior of $g_{x}$ and $g_{y}$ prevents mismatches such as $t\notin C_{g_{x}} $ but $t\in C_{g_{y}}$, where dividing by $g_{y}(s)-g_{y}(t)$ would not be possible when taking the limit. \par
	It also avoids the case $t\in(a_{n},b_{n})\subset C_{g_{x}}$ and $t\notin C_{g_{y}}$, where the derivative
	\[w'_{g_{x}}(t)=\lim\limits_{s\rightarrow b_{n}^{+}}\frac{w(s)-w(b_{n})}{g_{x}(s)-g_{x}(b_{n})}=\lim\limits_{s\rightarrow b_{n}^{+}}\frac{w(s)-w(b_{n})}{g_{y}(s)-g_{y}(b_{n})}\frac{g_{y}(s)-g_{y}(b_{n})}{g_{x}(s)-g_{x}(b_{n})}=w'_{g_{y}}(t)(g_{y})'_{g_{x}}(b_{n}),\]
	may not coincide with $w'_{g_{y}}(t)(g_{y})'_{g_{x}}(t)$. In this example, we have assumed that $b_{n}\notin C_{g_{y}}$ so that the quotient involving $g_{y}(s)-g_{y}(b_{n})$ is well\--defined.
\end{rem}

Assuming that $G$ satisfies the hypotheses of Proposition~\ref{rel_gx_gy} for all $x,y\in[0,L]$, and applying it to equation~\eqref{aux_v} yields
\begin{equation}\label{eqv_general}
	v''_{h_{t_{0}}}(x)=\frac{w'_{g_{x_{0}}}(t_{0})(g_{x_{0}})'_{g_{x}}(t_{0})}{c^{2}w(t_{0})}v(x)=\lambda (g_{x_{0}})'_{g_{x}}(t_{0})v(x).
\end{equation}

As we can see in
\eqref{eqv_general}, there is now a dependence of $x_0$ on the right\--hand side that we did not encounter in the previous cases. The expression relating $w'_{g_{x_0}}(t)$ and $w(t)$ through $\lambda$ leads to the same issue (it depends on $t_0$). These dependencies prevent a straightforward application of the method of separation of variables and highlight the complexity of equation \refeq{Gcalor} for a general derivator.

\section*{Declarations}

\subsection*{Funding}
F. Adrián F. Tojo was supported by grant PID2020-113275GB-I00 funded by Xunta de Galicia, Spain, project ED431C 2023/12; and by MCIN/AEI/10.13039/ 501100011033, Spain, and by “ERDF A way of making Europe” of the European Union.

\subsection*{Competing Interests}
The authors have no relevant financial or non-financial interests to disclose.

\subsection*{Author Contributions}
All authors contributed to conceptualization, research, writing and reviewing of the manuscript.

\bibliography{heat_equation}
\bibliographystyle{spmpsciper}
\end{document}